\documentclass[10pt]{article}%
\usepackage{amsmath}
\usepackage{amsfonts}
\usepackage{mathrsfs}
\usepackage{amssymb}
\usepackage[linkcolor=black,anchorcolor=black,citecolor=black]{hyperref}
\usepackage{graphicx}
\usepackage[body={15.5cm,21cm}, top=3cm]{geometry}%
\providecommand{\U}[1]{\protect\rule{.1in}{.1in}}
\providecommand{\U}[1]{\protect \rule{.1in}{.1in}}
\newtheorem{theorem}{Theorem}[section]

\newtheorem{definition}[theorem]{Definition}
\newtheorem{example}[theorem]{Example}

\newtheorem{lemma}[theorem]{Lemma}

\newtheorem{remark}[theorem]{Remark}

\newenvironment{proof}[1][Proof]{\noindent \textbf{#1.} }{\  \rule{0.5em}{0.5em}}
\begin{document}

\title{Ergodic BSDEs driven by $G$-Brownian motion and  applications}
\author{Mingshang Hu \thanks{Zhongtai Securities Institute of Finance, Shandong University,
humingshang@sdu.edu.cn. Research supported by  NSF (No.11671231, 11201262 and 11301068) and Shandong Provincial NSF (No.BS2013SF020 and ZR2014AP005)}
\and Falei Wang\thanks{Zhongtai Securities Institute of Finance and Institute for Advanced Research, Shandong University,
flwang2011@gmail.com. Research supported by  the NSF (No.11601282), Shandong Provincial NSF (No.ZR2016AQ10) and Fundamental Research Funds of Shandong University (No.2015GN023).
Hu and Wang's research was
partially supported by NSF (No.10921101) and by the 111
Project (No.B12023)}}
\date{}

\maketitle

\begin{abstract}
The present paper considers a new kind of backward stochastic differential equations (BSDEs) driven by
$G$-Brownian motion, which is called ergodic $G$-BSDEs. Firstly, the well-posedness of $G$-BSDEs with infinite horizon is  given
by a new linearization method. Then, the Feynman-Kac formula for fully nonlinear  elliptic partial differential equations (PDEs) is established.
Moreover,  a new probabilistic approach is introduced  to prove the uniqueness of viscosity solution to elliptic PDEs in the whole space.
Finally, we obtain the existence  of solution to $G$-EBSDE and  some  applications are also stated.
\end{abstract}


\textbf{MSC-classification}: 60H10, 60H30
\section{Introduction}
In 1990, Pardoux and Peng \cite{PP90}   established the existence and uniqueness theorem
for  nonlinear BSDEs, which  generalize the linear ones of
Bismut \cite{Bismut}.  After that, the researchers made great progress in this field.
In particular, the BSDEs theory provides a powerful tool for the study of mathematical
finance (see \cite{CHMP,EPQ}), stochastic control (see \cite{Peng1993}) and PDEs  (see \cite{PE, PP2}).

It is well known that  BSDEs with a deterministic terminal time provide a probabilistic
representation for solutions to quasi-linear parabolic PDEs, whereas the
BSDEs with a random terminal time are connected with quasi-linear elliptic PDEs (see \cite{BH,FH1,P1991,RM}). The BSDEs with infinite horizon
can be seen as a special case of BSDEs with a random terminal time. Based on it,
 Fuhrman, Hu and Tessitore \cite{FH} (see also \cite{CFa,CH,DH,R} for more details) introduced the following Markovian ergodic BSDE (EBSDE):
\begin{align*}
Y_{s}^{x}   =Y^x_T+\int_{s}^{T}[f(X_{r}^{x},Z_{r}^{x})-\lambda]dr -\int_{s}^{T}Z_{r}^{x}dW_{r},
\end{align*}
where $(W_r)_{r\geq 0}$ is a cylindrical Wiener process in a Hilbert space and $X^x$ is the solution to a
forward stochastic differential equation starting at $x$ and taking values in a Banach space.  In this equation,
the constant $\lambda$ is the ``ergodic cost'', which provides an efficient alternative tool  for the study of optimal control
problems with ergodic cost functionals (see also \cite{AL, BF}).
Moreover, by virtue of a EBSDE approach, Hu, Madec and Richou \cite{H} (see also \cite{HM}) studied the large time asymptotics of mild solutions to semi-linear PDEs under the so called weak dissipative assumptions  (in infinite dimension). In particular, they  also  gave an explicit  rate of convergence.
Based on the randomization approach introduced by \cite{KP}, Cosso, Fuhrman and Pham \cite{CF} obtained the long-time behavior of solutions to fully nonlinear HJB equations under dissipativity  conditions, where the diffusion term may be degenerate.

Recently, Peng introduced a sublinear expectation--$G$-expectation
theory, which non-trivially generalizes the classical  case  (see \cite {P10, PengICM2010} and the references therein).
Under the $G$-expectation framework, the $G$-Brownian motion and the corresponding
stochastic calculus of It\^{o}'s type are also established. Moreover, the existence and uniqueness theorem
of  $G$-BSDEs and nonlinear Feynman-Kac formula for fully nonlinear  PDEs are also  obtained in  \cite{HJPS,HJPS1}(see \cite{HL} for further research).
In a different setting,  Soner, Touzi and Zhang \cite{STZ} established the the so-called 2BSDEs theory, which shares many similarities with $G$-BSDEs.
For more research on this topic, we refer the reader to \cite{PX} and the references therein.

The present paper is devoted to studying the following type of Markovian BSDE driven by $G$-Brownian motion with infinite horizon, which is  called $G$-EBSDE:
for all $0\leq s\leq T<\infty$,
\begin{align}\label{HM12}
Y_{s}^{x}   =Y^x_T+\int_{s}^{T}[f(X_{r}^{x},Z_{r}^{x})+\gamma^1\lambda]dr+\int_{s}^{T}[g_{ij}(X_{r}^{x}%
,Z_{r}^{x})+\gamma^2_{ij}\lambda]d\langle B^i,B^j\rangle_{r} -\int_{s}^{T}Z_{r}^{x}dB_{r}-(K_{T}^{x}-K_{s}^{x}),
\end{align}
where  $\gamma^1$ is a fixed constant and  $\gamma^2 $ is a given $d\times d$ symmetric matrix  satisfying $\gamma^1+2G(\gamma^2)<0$, $(B_t)_{t\geq 0}$ is a  $d$-dimensional  $G$-Brownian motion and $X^x$ is the solution to a stochastic differential equation driven by $G$-Brownian motion starting at $x$.
Our aim is to find a quadruple $(Y, Z, K, \lambda)$ satisfying $G$-EBSDEs \eqref{HM12}, where $Y,Z$ are integrable processes in the $G$-expectation space,  $K$ is a decreasing $G$-martingale and
$\lambda$ is a real number.

For this purpose,  we firstly introduce a new kind of  linearization method to show
that the BSDE driven by $G$-Brownian motion with infinite horizon has a unique solution under some certain  conditions. Note that the  linearization methods in  \cite{BH} and \cite{HJPS1} cannot be  applied directly to deal with this problem due to the structure of $G$-expectation space.
In addition,  the comparison theorem for $G$-BSDE with infinite horizon is also obtained.
Then, we establish the fully nonlinear Feynman-Kac formula for elliptic PDEs and introduce a new probabilistic method to tackle the uniqueness of viscosity solution to elliptic PDEs in $\mathbb{R}^n$, which improves the one in \cite{PE}. Finally, we  prove that the $G$-EBSDE \eqref{HM12} has a solution $(Y^x, Z^x, K^x, \lambda)$.
The $G$-EBSDE \eqref{HM12}
provides an alternative approach for  the study of  the following  ergodic elliptic PDEs:
\begin{align*}
&G(H(D_{x}^{2}{v},D_{x}{v},\lambda,x))+\langle
b(x),D_{x}{v}\rangle+f(x,D_{x}{v}\sigma(x))+\gamma^1\lambda=0,
\end{align*}%
which is a  completely new  fully nonlinear PDE. Moreover, with the help of $G$-EBSDEs theory, we could study  the large time behaviour of solutions to fully nonlinear PDE
and optimal ergodic control problems under model uncertainty. Indeed, $G$-EBSDEs theory provides a potential method to study ergodic problems in the nonlinear expectation framework, see \cite{HL1}.

The paper is organized as follows. In  section 2, we present some preliminaries
for $G$-BSDEs. The  existence and uniqueness theorem for  $G$-BSDEs with infinite horizon is established in section 3.
In section 4, we obtain the fully nonlinear Feynman-Kac formula for   elliptic PDEs. Section 5 is devoted
to the study of $G$-EBSDEs and  some applications are stated in section 6.

\section{Preliminaries}

The main purpose of this section is to recall some basic notions and results of $G$-expectation theory, which are needed in the sequel. The readers may refer to \cite{HJPS},
\cite{P07a}, \cite{P08a} and \cite{P10} for more details.

Let $\Omega=C_{0}([0,\infty);\mathbb{R}^{d})$ be the space of all
$\mathbb{R}^{d}$-valued continuous functions on $[0,\infty)$ starting from the origin,  endowed
with the distance
$$
\rho(\omega^1, \omega^2):=\sum^\infty_{N=1} 2^{-N} [(\max_
{t\in[0,N]} | \omega^1_t-\omega^2_t|)  \wedge 1],
$$
and $B$ be the canonical
process. For each $T>0$, denote
\[
L_{ip} (\Omega_T):=\{ \varphi(B_{t_{1}},...,B_{t_{n}}):n\geq1,t_{1}%
,...,t_{n}\in\lbrack0,T],\varphi\in C_{b.Lip}(\mathbb{R}^{d\times n})\}, \ L_{ip} (\Omega):=\underset{T}{\cup}L_{ip} (\Omega_T),
\]
where $C_{b.Lip}(\mathbb{R}^{ n})$ is the space of all
bounded Lipschitz functions  on $\mathbb{R}^n$.

Let  $\mathbb{S}_{d}$ be the space of all $d\times d$ symmetric matrices. For each given monotonic and sublinear function
$G:\mathbb{S}_{d}\rightarrow\mathbb{R}$, Peng constructed a sublinear expectation
space $(\Omega,L_{ip}(\Omega),\hat{\mathbb{E}},(\hat{\mathbb{E}}_t)_{t\geq 0})$ called $G$-expectation space.
Indeed,  for each $\xi\in L_{ip}(\Omega)$ with the  form of
\begin{equation*}
\xi(\omega)=\varphi(\omega_{t_{1}},\omega_{t_{2}},\cdots,\omega_{t_{k}}),\  \
0=t_{0}<t_{1}<\cdots<t_{k}<\infty,
\end{equation*}
we  define the conditional  $G$-expectation by
\begin{equation*}
\mathbb{\hat{E}}_{t}[\xi]:=u_{i}(t,\omega_t;\omega_{t_{1}},\cdots,\omega_{
t_{i-1}})
\end{equation*}
for each $t\in \lbrack t_{i-1},t_{i})$, $i=1,\ldots,k$. Here,  the function $u_{i}(t,x;x_{1},\cdots,x_{i-1})$ parameterized by $(x_{1},\cdots,x_{i-1})\in \mathbb{R}^{(i-1)\times d}$
is the solution of the following $G$-heat equation:
\begin{equation*}
\partial_{t}u_{i}(t,x;x_{1},\cdots,x_{i-1})+G(D^2_{x}u_{i}(t,x;x_{1},\cdots,x_{i-1}))=0, \ \ (t,x)\in [t_{i-1},t_{i})\times \mathbb{R}^d
\end{equation*}
with terminal conditions
\begin{equation*}
u_{i}(t_{i},x;x_{1},\cdots,x_{i-1})=u_{i+1}(t_{i},x;x_{1},\cdots, x_{i-1},x),
\, \, \hbox{for $i<k$}
\end{equation*}
and $u_{k}(t_{k},x;x_{1},\cdots,x_{k-1})=\varphi (x_{1},\cdots, x_{k-1},x)$.
The $G$-expectation of $\xi$ is defined by $\mathbb{\hat{E}}[\xi]=%
\mathbb{\hat{E}}_{0}[\xi]$.
In this space the corresponding
canonical process $B_t(\omega) = \omega_t$
 is called $G$-Brownian motion.

 Denote by $L_{G}^{p}(\Omega)$   the completion of
$L_{ip} (\Omega)$ under the norm $|\mathbb{\hat{E}}[|\cdot|^{p}]|^{1/p}$ for each  $p\geq1$.
Denis et al. \cite{DHP11}
proved that the completions of $C_{b}(\Omega)$ (the set of all bounded
continuous functions on $\Omega$) and $L_{ip} (\Omega)$  are the same.
Similarly, we can define $L_{G}^{p}(\Omega_T)$ for each $T>0$. In this paper, we shall only consider non-degenerate $G$-Brownian motion,
i.e., there exist some constants $0<\underline{\sigma}^2\leq \bar{\sigma}^2<\infty $ such that, for any $A\geq B$ \[
 \frac{1}{2}\underline{\sigma}^2tr[A-B]\leq G(A)-G(B)\leq \frac{1}{2}\bar{\sigma}^2tr[A-B].\]
\begin{theorem}[\cite{DHP11,HP09}]
\label{the2.7}  There exists a weakly compact set  $\mathcal{P}$ of probability
measures on $(\Omega,\mathcal{B}(\Omega))$
  such that
\[
\mathbb{\hat{E}}[\xi]=\sup_{P\in\mathcal{P}}E_{P}[\xi]\ \ \text{for
all}\ \xi\in  {L}_{G}^{1}{(\Omega)}.
\]
$\mathcal{P}$ is called a set that represents $\mathbb{\hat{E}}$.
\end{theorem}

Let $\mathcal{P}$ be a weakly compact set that represents $\mathbb{\hat{E}}$.
Then we define the following capacity%
\[
c(A):=\sup_{P\in\mathcal{P}}P(A),\ A\in\mathcal{B}(\Omega).
\]
A set $A\subset\mathcal{B}(\Omega)$ is polar if $c(A)=0$.  A
property holds $``quasi$-$surely''$ (q.s.) if it holds outside a
polar set. In the following, we do not distinguish between two random
variables $X$ and $Y$ if $X=Y$ q.s..

\begin{definition}
\label{def2.6} Let $M_{G}^{0}(0,T)$ be the collection of processes in the
following form: for a given partition $\{t_{0},\cdot\cdot\cdot,t_{N}\}$ of $[0,T]$,
\[
\eta_{t}(\omega)=\sum_{j=0}^{N-1}\xi_{j}(\omega)\mathbf{1}_{[t_{j},t_{j+1})}(t),
\]
where $\xi_{i}\in L_{ip}(\Omega_{t_{i}})$, $i=0,1,2,\cdot\cdot\cdot,N-1$. For each
$p\geq1$,  denote by  $M_{G}^{p}(0,T)$ the completion
of $M_{G}^{0}(0,T)$ under the norm $|\mathbb{\hat{E}}[\int_{0}^{T}|\eta_{s}|^{p}ds]|^{1/p}$.
\end{definition}

For each $1\leq i, j \leq d$,   denote by $\langle B^i, B^j\rangle$   the mutual variation process.
Then for two processes $ \eta\in M_{G}^{2}(0,T)$ and $ \xi\in M_{G}^{1}(0,T)$,
the $G$-It\^{o} integrals  $\int\eta_sdB^i_s$ and $\int\xi_sd\langle
B^i,B^j\rangle_s$  are well defined, see  Li-Peng \cite{L-P} and Peng \cite{P10}.
Let $S_{G}^{0}(0,T)=\{h(t,B_{t_{1}\wedge t},\cdot\cdot\cdot,B_{t_{n}\wedge
t}):t_{1},\ldots,t_{n}\in\lbrack0,T],h\in C_{b,Lip}(\mathbb{R}^{n+1})\}$. For each $p\geq1$ and $\eta\in S_{G}^{0}(0,T)$, we set $\Vert\eta\Vert_{S_{G}^{p}}=|
\mathbb{\hat{E}}[\sup_{t\in\lbrack0,T]}|\eta_{t}|^{p}]|^{\frac{1}{p}}$ and
denote by $S_{G}^{p}(0,T)$ the completion of $S_{G}^{0}(0,T)$ under the norm
$\Vert\cdot\Vert_{S_{G}^{p}}$.

Now, consider the following type of $G$-BSDEs in a finite interval $[0,T]$ (in this paper we always use Einstein convention):%
\begin{align}\label{e3}
Y_{t}    =\xi+\int_{t}^{T}f(s,Y_{s},Z_{s})ds+\int_{t}^{T}g_{ij}(s,Y_{s}%
,Z_{s})d\langle B^{i},B^{j}\rangle_{s} -\int_{t}^{T}Z_{s}dB_{s}-(K_{T}-K_{t}), %
\end{align}
where

\[
f(t,\omega,y,z),g_{ij}(t,\omega,y,z):[0,\infty)\times\Omega\times
\mathbb{R}\times\mathbb{R}^{d}\rightarrow\mathbb{R}%
\]
satisfy the following properties:
\begin{description}
\item[(H1)] There exists a constant $\beta>0$ such that for any $y,z$,
$f(\cdot,\cdot,y,z),g_{ij}(\cdot,\cdot,y,z)\in M_{G}^{2+\beta}(0,n)$ for each $n>0$;

\item[(H2)] There exists a constant $L_1>0$ such that
\[
|f(t,y,z)-f(t,y^{\prime},z^{\prime})|+\sum_{i,j=1}^{d}|g_{ij}(t,y,z)-g_{ij}%
(t,y^{\prime},z^{\prime})|\leq L_1(|y-y^{\prime}|+|z-z^{\prime}|).
\]
\end{description}

For simplicity, we denote by $\mathfrak{S}_{G}^{2}(0,T)$ the collection
of processes $(Y,Z,K)$ such that $Y\in S_{G}^{2}(0,T)$, $Z\in
M_{G}^{2}(0,T;\mathbb{R}^{d})$, $K$ is a decreasing $G$-martingale with
$K_{0}=0$ and $K_{T}\in L_{G}^{2}(\Omega_{T})$.

\begin{theorem}[\cite{HJPS}]
\label{the4.1}  Assume that $\xi\in L_{G}^{2+\beta}(\Omega_{T})$
and $f$, $g_{ij}$ satisfy \emph{(H1)}, \emph{(H2)} for some $\beta>0$. Then equation
\eqref{e3} has a unique solution $(Y,Z,K)\in \mathfrak{S}_{G}^{2}(0,T)$.
\end{theorem}

We have the following estimates.

\begin{theorem}[\cite{HJPS}]
\label{pro3.5}  Let $\xi^{l}\in L_{G}^{2+\beta}(\Omega_{T})$ ,
$l=1,2$ and $f^{l}$, $g_{ij}^{l}$ satisfy \emph{(H1)},\emph{(H2)} for some $\beta>0$.
Assume that $(Y^{l},Z^{l},K^{l})\in\mathfrak{S}_{G}^{2}(0,T)$ is the solution of equation \eqref{e3} corresponding to the data
($\xi^{l}$ $f^{l}$,$g_{ij}^{l}$). Set $\hat{Y}_{t}=Y_{t}^{1}-Y_{t}^{2}%
,\hat{Z}_{t}=Z_{t}^{1}-Z_{t}^{2}$.
Then there exists a constant $C$ depending on
$T$, $G$, $L_1$ such that%
\begin{align*}
&\mathbb{\hat{E}}[\sup_{t\in\lbrack0,T]}|\hat{Y}_{t}|^{2}]    \leq
C\{ \mathbb{\hat{E}}[\sup_{t\in\lbrack0,T]}%
\mathbb{\hat{E}}_{t}[|\hat{\xi}|^{2}]]  +\mathbb{\hat{E}}[\sup_{t\in\lbrack0,T]}\mathbb{\hat{E}}_{t}[(\int_{0}%
^{T}\hat{h}_{s}ds)^{2}]]\},\\
&\mathbb{\hat{E}}[\int_{0}^{T}|\hat{Z}_{s}|^2ds]\leq
C\{ \Vert \hat{Y}\Vert_{{S}^{2}_G}^{2}+\Vert \hat {Y}%
\Vert_{{S}^{2}_G}\sum_{l=1}^{2}[||Y^{l}||_{%
{S}^{2}_G}+||\int_{0}^{T}h_{s}^{l,0}ds||_{%
L^2_G}]\},
\end{align*}
where $\hat{\xi}=\xi^{1}-\xi^{2}$, $\hat{h}_{s}=|f^{1}(s,Y_{s}^{2},Z_{s}%
^{2})-f^{2}(s,Y_{s}^{2},Z_{s}^{2})|+\sum_{i,j=1}^{d}|g_{ij}^{1}(s,Y_{s}%
^{2},Z_{s}^{2})-g_{ij}^{2}(s,Y_{s}^{2},Z_{s}^{2})|$ and $h_{s}^{l,0}=|f^l(s,0,0)|+|g^l_{ij}(s,0,0)|$.
\end{theorem}

Note that the estimate for $Z$ is different from the classical case because of the existence of the decreasing $G$-martingale $K$.
We also have the  explicit solutions of linear $G$-BSDEs.
For convenience, assume $d=1$. Consider the following linear $G$-BSDE in finite horizon $[0,T]$:
\begin{equation}
Y_{t}=\xi+\int_{t}^{T}f_{s}ds+\int_{t}^{T}g_{s}d\langle B\rangle_{s}-\int%
_{t}^{T}Z_{s}dB_{s}-(K_{T}-K_{t}), \label{LBSDE1}%
\end{equation}
where $f_{s}=a_{s}Y_{s}+b_{s}Z_{s}+m_{s}$, $g_{s}=c_{s}Y_{s}+d_{s}Z_{s}+n_{s}$
with bounded
processes $(a_{s})_{s\in\lbrack0,T]}$, $(b_{s})_{s\in\lbrack0,T]}$,
$(c_{s})_{ s\in\lbrack0,T]}$, $(d_{s})_{s\in\lbrack0,T]}\in M_{G}^{p}(0,T)$ and
$(m_{s})_{s\in\lbrack0,T]}$, $(n_{s})_{s\in\lbrack0,T]}\in M_{G}^{p}(0,T)$, $\xi\in L_{G}^{p}(\Omega_{T})$ for some $p>1$.

Then we construct an auxiliary extended
$\tilde{G}$-expectation space $(\tilde{\Omega},L_{\tilde{G}}^{1}%
(\tilde{\Omega}),\mathbb{\hat{E}}^{\tilde{G}})$ with $\tilde{\Omega}=C_{0}([0,\infty),\mathbb{R}^{2})$ and%
\[
\tilde{G}(A)=\frac{1}{2}\sup_{\underline{\sigma}^{2}\leq v\leq\bar{\sigma}%
^{2}}\mathrm{tr}\left[  A\left[
\begin{array}
[c]{cc}%
v & 1\\
1 & v^{-1}%
\end{array}
\right]  \right]  ,\ A\in\mathbb{S}_{2}.
\]
Let $(B_{t},\tilde{B}_{t})_{t\geq 0}$ be the canonical process in the extended space.

Suppose $\{X_{t}\}_{t\in\lbrack0,T]}$  is the solution of the following $\tilde{G}$-SDE:%
\begin{equation}
X_{t}=1+\int_{0}^{t}a_{s}X_{s}ds+\int_{0}^{t}c_{s}X_{s}d\langle B\rangle
_{s}+\int_{0}^{t}d_{s}X_{s}dB_{s}+\int_{0}^{t}b_{s}X_{s}d\tilde{B}_{s}.
\label{LSDE2}%
\end{equation}
It is easy to verify that%
\begin{equation}
X_{t}=\exp(\int_{0}^{t}(a_{s}-b_{s}d_{s})ds+\int_{0}^{t}c_{s}d\langle
B\rangle_{s})\mathcal{E}_{t}^{B}\mathcal{E}_{t}^{\tilde{B}}, \label{LSDE3}%
\end{equation}
where $\mathcal{E}_{t}^{B}=\exp(\int_{0}^{t}d_{s}dB_{s}-\frac{1}{2}\int%
_{0}^{t}d_{s}^{2}d\langle B\rangle_{s})$, $\mathcal{E}_{t}^{\tilde{B}}%
=\exp(\int_{0}^{t}b_{s}d\tilde{B}_{s}-\frac{1}{2}\int_{0}^{t}b_{s}^{2}%
d\langle\tilde{B}\rangle_{s})$.

\begin{lemma}[\cite{HJPS1}]
\label{the5.2} In the extended $\tilde{G}$-expectation space, the solution of
the $G$-BSDE \eqref{LBSDE1} can be represented as%
\begin{equation*}
Y_{t}=(X_{t})^{-1}\mathbb{\hat{E}}_{t}^{\tilde{G}}[X_{T}\xi+\int_{t}^{T}%
m_{s}X_{s}ds+\int_{t}^{T}n_{s}X_{s}d\langle B\rangle_{s}],
\end{equation*}
where $\{X_{t}\}_{t\in\lbrack0,T]}$ is the solution of the $\tilde{G}$-SDE
\eqref{LSDE2}.
Moreover,
\[
(X_{t})^{-1}\mathbb{\hat{E}}_{t}^{\tilde{G}}[X_{T}K_{T}-\int_{t}^{T}a_{s}K_{s}X_{s}%
ds-\int_{t}^{T}c_{s}K_{s}X_{s}d\langle
B\rangle_{s}]=K_{t}.
\]
\end{lemma}

The following estimate is important for our future discussions, whose proof will be given in the appendix.
\begin{lemma}
\label{my9}
Suppose the processes $(Y,Z,K)\in \mathfrak{S}_{G}^{2}(0,T)$ is the solution to the following equation
\[
Y_{t}=\xi+\int_{t}^{T}f_{s}ds+\int_{t}^{T}g_{s}d\langle B\rangle_{s}-\int%
_{t}^{T}Z_{s}dB_{s}-(K_{T}-K_{t})+(\bar{K}_T-\bar{K}_t),
\]
where $\bar{K}_t\in L^p_G(\Omega_t)$ is a decreasing $G$-martingale  for some $p>1$.
Moreover, $\xi$ is bounded by some constant $\rho_1$, $m_s, n_s$ are bounded by some constant $\rho_2$ and
 $a_s+\overline{\sigma}^2c_s\leq -\rho_3$ for some constant $\rho_3>0$.
Then
\[
Y_t\leq \rho_1\exp(-\rho_3(T-t))+\frac{1+\overline{\sigma}^2}{\rho_3}\rho_2.
\]
If we further assume that $\bar{K}_t=0$, then \[
|Y_t|\leq \rho_1\exp(-\rho_3(T-t))+\frac{1+\overline{\sigma}^2}{\rho_3}\rho_2.
\]
\end{lemma}

\section{$G$-BSDEs with infinite horizon}
For simplicity, we consider the $G$-expectation space $%
(\Omega,L_{G}^{1}(\Omega),\mathbb{\hat{E}})$ with $%
\Omega=C_{0}([0,\infty),\mathbb{R})$ and $\bar{\sigma}^{2}=\mathbb{\hat{E%
}}[B_{1}^{2}]\geq-\mathbb{\hat{E}}[-B_{1}^{2}]=\underline{\sigma}^{2}>0$.
But our results and methods still hold for the case $d>1$.

This section is devoted to studying the following type of BSDEs driven by $G$-Brownian motion with infinite horizon,
\begin{align} \label{HM}
Y_t=Y_T+\int^T_tf(s,Y_s,Z_s)ds+\int^T_tg(s,Y_s,Z_s)d\langle B\rangle_s-\int^T_tZ_sdB_s-(K_T-K_t), \ \ 0\leq t\leq T<\infty.
\end{align}

In the rest of this section we shall make use of the following assumptions on
the generators of $G$-BSDEs.
\begin{description}
\item[(H3)]  There exists a constant $\mu>0$  such that $(f(t,\omega,y,z)-f(t,\omega,y^{\prime},z))(y-y^{\prime})+2G((g(t,\omega,y,z)-g(t,\omega,y^{\prime},z))(y-y^{\prime}))\leq -\mu|y-y^{\prime}|^2 $.
\item[(H4)]
 $|f(s,0,0)|+\bar{\sigma}^2|g(s,0,0)|\leq L_2$ for some constant $L_2$.
\end{description}
\begin{definition}
\label{def3.1} A triplet of processes $(Y,Z,K)$ is called a solution
of equation \eqref{HM} if the following
properties hold:

\begin{description}
\item[(a)] $(Y,Z,K)\in \mathfrak{S}_{G}^{2}(0,\infty)$, where
$\mathfrak{S}_{G}^{2}(0,\infty)=\underset{T}{\cap}\mathfrak{S}_{G}^{2}(0,T)$;

\item[(b)] $Y_{t}=Y_T+\int_{t}^{T}f(s,Y_{s},Z_{s})ds+\int^T_tg(s,Y_s,Z_s)d\langle B\rangle_s-%
\int_{t}^{T}Z_{s}dB_{s}-(K_{T}-K_{t}), \ 0\leq t\leq T<\infty$.
\end{description}
\end{definition}

In this paper we only consider the case that $Y$ component of the solution to $G$-BSDE \eqref{HM} is bounded.
Indeed, $G$-BSDE \eqref{HM} may have more than one  solution.
\begin{example}{\upshape Taking $f(s,y,z)=-y$ and $g=0$,
one can easily show that $(ce^{t},0,0)$ is a solution to equation \eqref{HM} for each constant $c$. However, it has a unique bounded solution $(0,0,0)$.
}
\end{example}
\begin{remark}{\upshape Remark that (H3) is necessary to ensure the uniqueness of solution to equation \eqref{HM}.
 For example, taking $f(s,y,z)=y$ and $g=0$,
one can easily check that $(ce^{-t},0,0)$ is a bounded solution for each constant $c$.
}
\end{remark}
\begin{remark}{\upshape In order to state the main idea, we content ourselves with the case that $f(s,0,0)$ and $g(s,0,0)$ are bounded. Indeed, (H4) can be  weakened by a slightly more involved estimates (see, e.g. \cite{RM}). }
\end{remark}

The following result will be frequently used in this paper, which can be seen as a new version of linearization method for $G$-BSDEs.
\begin{lemma}\label{HW4}
For each given $\varepsilon>0$, there exist four bounded processes $a^{\varepsilon}_s(y,y^{\prime},z), b^{\varepsilon}_s(z,z^{\prime},y^{\prime})$, $c^{\varepsilon}_s(y,y^{\prime},z)$, $d^{\varepsilon}_s(z,z^{\prime},y^{\prime})$, such that
\begin{align*}
&a^{\varepsilon}_s(y,y^{\prime},z)+2G(c^{\varepsilon}_s(y,y^{\prime},z))\leq -\mu,\\
&
|f(s,y,z)-f(s,y^{\prime},z^{\prime})-a^{\varepsilon}_s(y,y^{\prime},z)(y-y^{\prime})-b^{\varepsilon}_s(z,z^{\prime},y^{\prime}))(z-z^{\prime})
|\leq 4L_1\varepsilon,\\&
|g(s,y,z)-g(s,y^{\prime},z^{\prime})-c^{\varepsilon}_s(y,y^{\prime},z)(y-y^{\prime})-d^{\varepsilon}_s(z,z^{\prime},y^{\prime}))(z-z^{\prime})
|\leq 4L_1\varepsilon.
\end{align*}
Moreover, for each $T>0$ and $Y,Y^{\prime},Z,Z^{\prime}\in M^2_G(0,T)$,
$a^{\varepsilon}_s(Y_s,Y^{\prime}_s,Z_s), b^{\varepsilon}_s(Z_s,Z^{\prime}_s,Y^{\prime}_s)$, $c^{\varepsilon}_s(Y_s,Y^{\prime}_s,Z_s)$,
$d^{\varepsilon}_s(Z_s,Z^{\prime}_s,Y^{\prime}_s)$ are in $M^2_G(0,T).$
\end{lemma}
\begin{proof}
Denote: \begin{align*}
&a^{\varepsilon}_s(y,y^{\prime},z):=l(y,y^{\prime},z)\frac{f(s,y,z)-f(s,y^{\prime},z)}{y-y^{\prime}}-\frac{\mu}{1+\underline{\sigma}^2}(1-l(y,y^{\prime},z)),\\
&c^{\varepsilon}_s(y,y^{\prime},z):=l(y,y^{\prime},z)\frac{g(s,y,z)-g(s,y^{\prime},z)}{y-y^{\prime}}-\frac{\mu}{1+\underline{\sigma}^2}(1-l(y,y^{\prime},z)),
\end{align*}
where $ l(y,y^{\prime},z)=\mathbf{1}_{|y-y^{\prime}|\geq \varepsilon}+ \frac{|y-y^{\prime}|}{\varepsilon}\mathbf{1}_{|y-y^{\prime}|< \varepsilon}$. It is obvious that
$a^{\varepsilon}_s(y,y^{\prime},z), c^{\varepsilon}_s(y,y^{\prime},z)$ are continuous functions in $(y,y^{\prime},z)$. Thus for each $T>0$, we conclude that $a^{\varepsilon}_s(Y_s,Y^{\prime}_s,Z_s)$, $c^{\varepsilon}_s(Y_s,Y^{\prime}_s,Z_s)$ are in $ M^2_G(0,T)$ for each  $Y,Y^{\prime},Z\in M^2_G(0,T)$.

From assumption (H3), we obtain that
\begin{align*}
a^{\varepsilon}_s(y,y^{\prime},z)+2G(c^{\varepsilon}_s(y,y^{\prime},z))
\leq& l(y,y^{\prime},z)(\frac{f(s,y,z)-f(s,y^{\prime},z)}{y-y^{\prime}}+2G(\frac{g(s,y,z)-g(s,y^{\prime},z)}{y-y^{\prime}}))\\&
+(1-l(y,y^{\prime},z))(-\frac{\mu}{1+\underline{\sigma}^2}+2G(-\frac{\mu}{1+\underline{\sigma}^2}))\\
\leq & -\mu.
\end{align*}
Note that $|a^{\varepsilon}_s|\leq L_1$ . Then by assumption (H2), we also derive that
\begin{align*}
&|f(s,y,z)-f(s,y^{\prime},z)-a^{\varepsilon}_s(y,y^{\prime},z)(y-y^{\prime})|\\&
\leq |f(s,y,z)-f(s,y^{\prime},z)|\mathbf{1}_{|y-y^{\prime}|< \varepsilon}+|a^{\varepsilon}_s(y,y^{\prime},z)(y-y^{\prime})|\mathbf{1}_{|y-y^{\prime}|< \varepsilon}\\
&\leq 2L_1|y-y^{\prime}|\mathbf{1}_{|y-y^{\prime}|< \varepsilon}\leq 2L_1\varepsilon.
\end{align*}
Finally, we set
\begin{align*}
&b^{\varepsilon}_s(z,z^{\prime},y^{\prime}):=l(z,z^{\prime},y^{\prime})\frac{f(s,y^{\prime},z)-f(s,y^{\prime},z^{\prime})}{z-z^{\prime}}
+L_1(1-l(z,z^{\prime},y^{\prime})),\\
&d^{\varepsilon}_s(z,z^{\prime},y^{\prime}):=l(z,z^{\prime},y^{\prime})\frac{g(s,y^{\prime},z)-g(s,y^{\prime},z^{\prime})}{z-z^{\prime}}
+L_1(1-l(z,z^{\prime},y^{\prime})).
\end{align*}
One can easily check that the last two inequalities also hold true.
\end{proof}

Now we  state the main result of this section, concerning the existence and uniqueness of
solutions of BSDE \eqref{HM}.
\begin{theorem}\label{HM3}
Let assumptions  \emph{(H1)-(H4)} hold. Then the $G$-BSDE \eqref{HM} has a unique solution $(Y, Z, K)$ belonging to $\mathfrak{S}_{G}^{2}(0,\infty)$  such that $Y$
is a bounded process.
\end{theorem}
\begin{proof}
{\bf Uniqueness:} Suppose that $(Y^1, Z^1,K^1)$ and $(Y^2, Z^2,K^2)$
are both solutions of the $G$-BSDE \eqref{HM}.
Set $(\hat{Y},\hat{Z})=(Y^1-Y^2,Z^1-Z^2)$. Since both $Y^1$ and $Y^2$ are bounded continuous processes, we can find  some  constant
$C>0$ such that $|\hat{Y}|\leq C$.
Then we have for any $T>0$,%
\[
\hat{Y}_{t}+K_{t}^{2}=\hat{Y}_T+K_{T}^{2}+\int_{t}^{T}\hat{f}_{s}ds+\int%
_{t}^{T}\hat{g}_{s}d\langle B\rangle_{s}-\int_{t}^{T}\hat{Z}_{s}dB_{s}%
-(K_{T}^{1}-K_{t}^{1}),
\]
where $\hat{f}_{s}=f(s,Y_{s}^{1},Z_{s}%
^{1})-f(s,Y_{s}^{2},Z_{s}^{2})$, $\hat{g}_{s}=g(s,Y_{s}^{1},Z_{s}%
^{1})-g(s,Y_{s}^{2},Z_{s}^{2})$.
From Lemma \ref{HW4}, for each given $\varepsilon>0$, we set $a^{\varepsilon}_s:=a^{\varepsilon}_s(Y^1_s,Y^2_s,Z^1_s)$.
Thus \[
f(s,Y_{s}^{1},Z_{s}%
^{1})-f(s,Y_{s}^{2},Z_{s}^{1})=a^{\varepsilon}_s\hat{Y}_s+f(s,Y_{s}^{1},Z_{s}%
^{1})-f(s,Y_{s}^{2},Z_{s}^{1})-a^{\varepsilon}_s\hat{Y}_s.
\]
Moreover, we can get $a^{\varepsilon}_s\in M_{G}^{2}(0,T)$.
Similarly, we can define $b_{s}^{\varepsilon},c_{s}^{\varepsilon}$ and $d_{s}^{\varepsilon}$.
Consequently,
\[
\hat{f}_{s}=a_{s}^{\varepsilon}\hat{Y}_{s}+b_{s}^{\varepsilon}\hat{Z}%
_{s}-m_{s}^{\varepsilon},\ \hat{g}_{s}=c_{s}^{\varepsilon}\hat{Y}%
_{s}+d_{s}^{\varepsilon}\hat{Z}_{s}-n_{s}^{\varepsilon},
\]
where $|m_{s}^{\varepsilon}|:=|\hat{f}_{s}-a_{s}^{\varepsilon}\hat{Y}_{s}-b_{s}^{\varepsilon}\hat{Z}%
_{s}|\leq4L_1\varepsilon$ and  $|n_{s}^{\varepsilon}%
|:=|\hat{g}_{s}-c_{s}^{\varepsilon}\hat{Y}_{s}-d_{s}^{\varepsilon}\hat{Z}%
_{s}|\leq4L_1\varepsilon$. Recalling Lemma \ref{my9} and letting $\varepsilon\rightarrow0$, we deduce that
\[
\hat{Y}_{t}\leq C\exp(-\mu(T-t)), \ \ \forall T>0, \ \ q.s..
\]
Therefore by sending $T$ to infinity yields that
$\forall t\geq 0$, $Y^1_t\leq Y^2_t,$ q.s.. By a similar analysis, we also have $Y^2_t\leq Y^1_t$, q.s..
Thus it follows from the continuity of $Y^1$ and $Y^2$ that $Y^1 =Y^2$, q.s..
Then recalling the uniqueness of solution to $G$-BSDE in finite horizon, we can also get
the uniqueness of $(Z,K)$, which is the desired result.

{\bf Existence:}
Denote by $(Y^n , Z^n, K^n)\in \mathfrak{S}_{G}^{2}(0,n)$ the unique solution of the
following $G$-BSDE in finite horizon:
\begin{align*}
Y^n_t=\int^n_tf(s,Y^n_s,Z^n_s)ds+\int^n_tg(s,Y^n_s,Z^n_s)d\langle B\rangle_s-\int^n_tZ^n_sdB_s-(K_n^n-K^n_t), \ \ 0\leq t\leq n.
\end{align*}
Using the same method as in the proof of uniqueness, we have
\[
Y^n_t=\int_{t}^{n}(f(s,0,0)+{f}_{s})ds+\int%
_{t}^{n}(g(s,0,0)+{g}_{s})d\langle B\rangle_{s}-\int_{t}^{n}{Z}^n_{s}dB_{s}%
-(K^n_{n}-K_{t}^n),
\]
where ${f}_{s}=f(s,Y_{s}^{n},Z_{s}%
^{n})-f(s,0,0)$, ${g}_{s}=g(s,Y_{s}^{n},Z_{s}%
^{n})-g(s,0,0)$. Then  for each $\varepsilon>0$, we can get
\[
{f}_{s}=a_{s}^{n,\varepsilon}Y^n_{s}+b_{s}^{n,\varepsilon}Z^n
_{s}-m_{s}^{n,\varepsilon},\ {g}_{s}=c_{s}^{n,\varepsilon}Y^n
_{s}+d_{s}^{n,\varepsilon}Z^n_{s}-n_{s}^{n,\varepsilon},
\]
where $|m_{s}^{n,\varepsilon}|\leq4L_1\varepsilon$ and  $|n_{s}^{n,\varepsilon}%
|\leq4L_1\varepsilon$.
By  Lemma \ref{my9}, we derive that
\begin{align*}
|Y^n_{t}|\leq \frac{L_2}{\mu}+4(1+\bar{\sigma}^2)\varepsilon\frac{L_1}{\mu}, \ \ q.s..
\end{align*}
Then letting $\varepsilon\rightarrow0$,
we can obtain that
\begin{align}\label{HM2}
|Y^n_{t}|\leq \frac{L_2}{\mu}, \ \ q.s..
\end{align}

Now we define $Y^n$, $Z^n$ and $K^n$ on the
whole time axis by setting
\[
Y^n_t=Z^n_t=0,\ K^n_t=K^n_n, \ \ \forall t>n.
\]

Fix $t\leq n\leq m$ and set $\tilde{Y} =Y^m-Y^n$, $\tilde{Z}=Z^m-Z^n$. As in the proof of uniqueness, we use the same kind of linearization.
Thus
\[
\tilde{Y}_{t}+K_{t}^{m}=K_{m}^{m}+\int_{t}^{m}\tilde{f}_{s}ds+\int%
_{t}^{m}\tilde{g}_{s}d\langle B\rangle_{s}-\int_{t}^{T}\tilde{Z}_{s}dB_{s}%
-(K_{m}^{n}-K_{t}^{n}),
\]
where $\tilde{f}_{s}=f(s,Y_{s}^{m},Z_{s}%
^{m})-f(s,Y_{s}^{n},Z_{s}^{n})+\mathbf{1}_{s>n}f(s,0,0)$, $\tilde{g}_{s}=g(s,Y_{s}^{m},Z_{s}%
^{m})-g(s,Y_{s}^{n},Z_{s}^{n})+\mathbf{1}_{s>n}g(s,0,0)$.
Then for each given $\varepsilon>0$, we have
\[
\tilde{f}_{s}=a_{s}^{m,n,\varepsilon}\tilde{Y}_{s}+b_{s}^{m,n,\varepsilon}\tilde{Z}%
_{s}-m_{s}^{m,n,\varepsilon}+\mathbf{1}_{s>n}f(s,0,0),\ \hat{g}_{s}=c_{s}^{m,n,\varepsilon}\tilde{Y}%
_{s}+d_{s}^{m,n,\varepsilon}\tilde{Z}_{s}-n_{s}^{m,n,\varepsilon}+\mathbf{1}_{s>n}g(s,0,0),
\]
where $|m_{s}^{m,n,\varepsilon}|\leq4L_1\varepsilon$ and  $|n_{s}^{m,n,\varepsilon}%
|\leq4L_1\varepsilon$. Therefore using the same strategy implies that
\begin{align}\label{HM1}
|\tilde{Y}_{t}|\leq \frac{L_2}{\mu}\exp(\mu t)(\exp(-\mu n)-\exp(-\mu m)), \ \ q.s..
\end{align}
Thus, we get for each $0<T\leq n\leq m$,
\[
\lim\limits_{m,n\rightarrow\infty}\mathbb{\hat{E}}[\sup\limits_{t\in[0,T]}|Y^n_t-Y^m_t|^2]=0.
\]
Consider the following $G$-BSDE in finite horizon $[0,T]$:
\[
Y^n_t=Y^n_T+\int^T_tf(s,Y^n_s,Z^n_s)ds+\int^T_tg(s,Y^n_s,Z^n_s)d\langle B\rangle_s-\int^T_tZ^n_sdB_s-(K_T^n-K^n_t).
\]
By  Theorem \ref{pro3.5},  we also conclude that
 \[
 \lim\limits_{m,n\rightarrow\infty}\|Z^n-Z^m\|_{M_G^2(0,T)}=0.
 \]
Consequently, there exist  two processes $(Y,Z)\in\mathcal{S}_G^2(0,\infty)\times M_G^2(0,\infty)$ such that
\[
\lim\limits_{n\rightarrow\infty}\mathbb{\hat{E}}[\sup\limits_{t\in[0,T]}|Y^n_t-Y_t|^2+\int^T_0|Z^n_t-Z_t|^2dt]=0.
\]
Moreover,  from equations \eqref{HM2} and \eqref{HM1}, we get that $|Y_t|\leq \frac{L_2}{\mu}$ and $|Y^n_t-Y_t|\leq \frac{L_2}{\mu}\exp(-\mu (n-t))$, q.s..

Denote
\[
K_t:=Y_t-Y_0+\int^t_0f(s,Y_s,Z_s)ds+\int^t_0g(s,Y_s,Z_s)d\langle B\rangle_s-\int^t_0Z_sdB_s.
\]
 Then we have $\hat{\mathbb{E}}[|K_{t}-{K}_{t}^{n}|^{2}]\rightarrow0$.
Moreover, $K$ is a $G$-martingale. Indeed,
for each $%
0\leq t<s$,
\begin{align*}
\hat{\mathbb{E}}[|\hat{\mathbb{E}}_{t}[K_{s}]-K_{t}|] & =\hat{\mathbb{E}}[|%
\mathbb{\hat{E}}_{t}[K_{s}]-\mathbb{\hat{E}}_{t}[{K}_{s}^{n}]+{K}%
_{t}^{n}-K_{t}|] \\
& \leq \hat{\mathbb{E}}[\mathbb{\hat{E}}_{t}[|K_{s}-{K}_{s}^{n}|]]+%
\hat{\mathbb{E}}[|{K}_{t}^{n}-K_{t}|] \\
& =\hat{\mathbb{E}}[|K_{s}-{K}_{s}^{n}|]+\hat{\mathbb{E}}[|{K}%
_{t}^{n}-K_{t}|]\rightarrow0.
\end{align*}
Thus we get $\mathbb{\hat{E}}_{t}[K_{s}]=K_{t}$,
which completes the proof.
\end{proof}
\begin{remark}{\upshape
The main difficulty to prove Theorem \ref{HM3} is the explicit solutions of linear $G$-BSDEs, which is different from the linear case.
Then  in the above proof we  introduce a new version of
linearization method  to
obtain the existence and uniqueness of $G$-BSDE with infinite horizon.
In particular, it also provides a new prior estimate for $G$-BSDEs (see equation \eqref{HM2}).}
\end{remark}

By the same way as in the proof of Theorem \ref{HM3}, we  also have the following comparison theorem.

\begin{theorem}[Comparison Theorem]\label{HM7}
Let $(Y^i,Z^i,K^i)$, $i=1,2$ be the solution of BSDE \eqref{HM} with generators $f^i$ and $g^i$ such that $Y^i$ is a bounded process. Moreover $f^i$ and $g^i$ satisfy assumptions
\emph{(H1)-(H4)}.
If  $f^1(s,Y^i_s,Z^i_s)-f^2(s,Y^i_s,Z^i_s)+2G(g^1(s,Y^i_s,Z^i_s)-g^2(s,Y^i_s,Z^i_s))\leq 0$ for some $i$, q.s.,  then for each $t$,
$Y^1_t\leq Y^2_t,$ q.s..
\end{theorem}

\section{Fully nonlinear Feynman-Kac formula for  elliptic PDEs}
In this section, we shall give the fully nonlinear Feynman-Kac Formula for $G$-BSDEs with infinite horizon. Let  $B_{t}=(B_{t}^{i})_{i=1}^{d}$ be the corresponding $d$-dimensional $G$-Brownian motion.
Consider the following type of $G$-FBSDEs with infinite horizon:%
\begin{align} \label{App1}
\begin{cases}
&X_{s}^{x}=x+\int^s_0b(X_{r}^{x})dr+\int^s_0h_{ij}(X_{r}^{x})d\langle
B^i,B^j\rangle_{r}+\int^s_0\sigma(X_{r}^{x})dB_{r},\\&
Y_{s}^{x}   =Y_{T}^{x}+\int_{s}^{T}f(X_{r}^{x}%
,Y_{r}^{x},Z_{r}^{x})dr+\int_{s}^{T}g_{ij}(X_{r}^{x}%
,Y_{r}^{x},Z_{r}^{x})d\langle B^i,B^j\rangle_{r}\\
& \ \ \ \ \ \ \ \ \ \ \ \  -\int_{s}^{T}Z_{r}^{x}dB_{r}-(K_{T}^{x}-K_{s}^{x}),
\end{cases}
\end{align}
where $b$, $h_{ij}:\mathbb{R}^{n}\rightarrow
\mathbb{R}^{n}$, $\sigma:\mathbb{R}^{n}\rightarrow
\mathbb{R}^{n\times d}$, $f$, $g_{ij}:$
$\mathbb{R}^{n}\times\mathbb{R}\times\mathbb{R}^{d}\rightarrow
\mathbb{R}$ are deterministic continuous functions. Consider also the following assumptions:
\begin{description}
\item[(B1)] $h_{ij}=h_{ji}$ and $g_{ij}=g_{ji}$ for $1\leq i,j\leq d$, $|f(x,0,0)|+2G(|g_{ij}(x,0,0)|)$ is bounded by some constant $\alpha$;
\item[(B2)] There exist some constants $L,\alpha_1>0$ and $\alpha_2>0$ such that%
\begin{align*}
&|b(x)-b(x^{\prime})|+\sum\limits_{i,j}|h_{ij}(x)-h_{ij}(x^{\prime
})|\leq
L|x-x^{\prime}|, \ \ |\sigma(x)-\sigma(x^{\prime})|\leq
\alpha_1|x-x^{\prime}|,\\
&  |f(x,y,z)-f(x^{\prime},y^{\prime},z^{\prime})|+
\sum\limits_{i,j}|g_{ij}(x,y,z)-g_{ij}(x^{\prime},y^{\prime},z^{\prime})|\\
&  \leq L(|x-x^{\prime}|+|y-y^{\prime
}|)+\alpha_2|z-z^{\prime}|.
\end{align*}
\item[(B3)] There exists a constant $\mu>0$  such that $(f(x,y,z)-f(x,y^{\prime},z))(y-y^{\prime})+2G((g_{ij}(x,y,z)-g_{ij}(x,y^{\prime},z))(y-y^{\prime}))\leq
-\mu|y-y^{\prime}|^2 $.
\item[(B4)]
$G(\sum\limits_{i=1}^n(\sigma_i(x)-\sigma_i(x^{\prime}))^T(\sigma_i(x)-\sigma_i(x^{\prime}))+2(\langle x-x^{\prime},h_{ij}(x)-h_{ij}(x^{\prime})\rangle)_{i,j=1}^d )+\langle x-x^{\prime},b(x)-b(x^{\prime})\rangle\leq -\eta|x-x^{\prime}|^2$ for some constant $\eta> 0$, where $\sigma_i$ is the $i$-th row of $\sigma$.
\item[(B5)] $\eta-(1+\bar{\sigma}^2)\alpha_1\alpha_2>0$.
\end{description}
By  Theorem \ref{HM3}, there exists a unique  solution $(X^x, Y^x,Z^x,K^x)$ to $G$-FBSDEs \eqref{App1} under (B1)-(B3). The assumptions (B4) and (B5) are called strong dissipativity assumptions and they ensures the ergodicity of
the diffusion process $X$ in the linear case  (see  \cite{CF},  \cite{FH} and  \cite{R}).

The following result is important in our future discussion.
\begin{lemma}\label{HW2}
Assume  $\tilde{X}$  is the solution of the following $\tilde{G}$-SDE:%
\begin{equation*}
\tilde{X}_{t}=1+\int_{0}^{t}d_{s}\tilde{X}_{s}dB_{s}+\int_{0}^{t}b_{s}\tilde{X}_{s}d\tilde{B}_{s},
\end{equation*}
 where  $(b_{s})_{s\in\lbrack0,\infty)}$,  $(d_{s})_{s\in\lbrack0,\infty)}$ are
in $M_{G}^{2}(0,T)$ for any $T>0$ and bounded by $\alpha_2$. Then the following properties hold:
\begin{description}
\item[(i)] $\hat{\mathbb{E}}^{\tilde{G}}[|X^x_t-X^{x^{\prime}}_t|\tilde{X}_t]\leq\exp(-\eta t+(1+\bar{\sigma}^2)\alpha_1\alpha_2t)|x-x^{\prime}|$;
\item[(ii)] there exists a constant $\bar{C}$ depending on $G,\alpha_1$, $\alpha_2$ and $\eta$, such that
\[\hat{\mathbb{E}}^{\tilde{G}}[|X^x_t|\tilde{X}_t]\leq \bar{C}(1+|x|), \ \forall t>0.
\]
\end{description}
\end{lemma}
\begin{proof}
Without loss of generality, assume $d=1$. It is obvious $\tilde{X}$ is a $\tilde{G}$-martingale.
Then
\[
\mathbb{\hat{E}}^{\tilde{G}}[|X^x_t-X^{x^{\prime}}_t|\tilde{X}_{t}]
\leq \mathbb{\hat{E}}^{\tilde{G}}[|X^x_t-X^{x^{\prime}}_t|^2\tilde{X}_{t}]^{\frac{1}{2}}\mathbb{\hat{E}}^{\tilde{G}}[\tilde{X}_{t}]^{\frac{1}{2}}
=\mathbb{\hat{E}}^{\tilde{G}}[|X^x_t-X^{x^{\prime}}_t|^2\tilde{X}_{t}]^{\frac{1}{2}}.
\]
Next we shall give the estimate of $|X^x_t-X^{x^{\prime}}_t|^2\tilde{X}_t$.
Set $\bar{X}_t=X^x_t-X^{x^{\prime}}_t$, $C=\eta-(1+\bar{\sigma}^2)\alpha_1\alpha_2$ and $\bar{\varphi}_s=\varphi(X^x_s)-\varphi(X^{x^{\prime}}_s)$ for $\varphi=b,h,\sigma$.
Applying the $G$-It\^{o} formula  yields that
\begin{align*}
&\exp(2Ct)|\bar{X}_t|^2\tilde{X}_t-|x-x^{\prime}|^2
\\&=
2C\int^t_0\exp(2Cs)|\bar{X}_s|^2\tilde{X}_sds+
2\int^t_0\exp(2Cs)\langle \bar{X}_t, \bar{b}_s\rangle\tilde{X}_s ds+
\int^t_0\exp(2Cs)\xi_s\tilde{X}_sd\langle B
\rangle_s
 \\&\ \ +M_t+2\int^t_0\exp(2Cs)\langle \bar{X}_s,\bar{\sigma}_s\rangle d_s\tilde{X}_sd\langle B\rangle_s +2\int^t_0\exp(2Cs)\langle\bar{X}_s, \bar{\sigma}_s\rangle b_s\tilde{X}_sds\\
&=
2C\int^t_0\exp(2Cs)|\bar{X}_s|^2\tilde{X}_sds+\Lambda_t^1+\Lambda_t^2+N_t+M_t,
\end{align*}
where
\begin{align*}
&\Lambda_t^1=2\int^t_0\exp(2Cs)[\langle \bar{X}_s,  \bar{b}_s\rangle+G(\xi_s)]\tilde{X}_s ds,\\
&\Lambda_t^2=2\int^t_0\exp(2Cs)\langle \bar{X}_s, \bar{\sigma}_s\rangle b_s\tilde{X}_sds
+2\int^t_0\exp(2Cs)\langle \bar{X}_s, \bar{\sigma}_s\rangle d_s\tilde{X}_sd\langle B\rangle_s,\\
&\xi_t=2\langle \bar{X}_t, \bar{ h}_t\rangle+|\bar{\sigma}_t|^2,\ \ N_t=\int^t_0\exp(2Cs)\xi_s\tilde{X}_sd\langle B\rangle_s-2\int^t_0\exp(2Cs)G(\xi_s)\tilde{X}_s ds,\\
&M_t=2\int^t_0\exp(2Cs)\langle \bar{X}_s,
\bar{\sigma}_s\rangle\tilde{X}_s d B_s+\int^t_0\exp(2Cs)|\bar{X}_s|^2d_s\tilde{X}_sd B_s+\int^t_0\exp(2Cs)|\bar{X}_s|^2b_s\tilde{X}_sd\tilde{B}_s.
\end{align*}
Then by assumption (B4), we obtain that
\begin{align*}
\Lambda_t^1\leq -2\eta \int^t_0 \exp(2Cs)|\bar{X}_s|^2\tilde{X}_sds, \
\Lambda_t^2\leq 2(1+\bar{\sigma}^2)\alpha_1\alpha_2\int^t_0 \exp(2Cs)|\bar{X}_s|^2\tilde{X}_sds.
\end{align*}
Note that  $N_t$ is a decreasing $G$-martingale and $N_t\leq 0$. Thus we conclude that
 \begin{align*}
 \exp(2Ct)|X^x_t-X^{x^{\prime}}_t|^2\tilde{X}_t\leq |x-x^{\prime}|^2+M_t.
\end{align*}
In sprit of $M_t$ is a symmetric $G$-martingale, we derive that
\[
\mathbb{\hat{E}}^{\tilde{G}}[\exp(2Ct)|X^x_t-X^{x^{\prime}}_t|^2\tilde{X}_t]\leq |x-x^{\prime}|^2.
\]
Consequently,
\begin{align*}
\mathbb{\hat{E}}^{\tilde{G}}[|X^x_t-X^{x^{\prime}}_t|\tilde{X}_t]\leq \exp(-\eta t+(1+\bar{\sigma}^2)\alpha_1\alpha_2t)|x-x^{\prime}|,
\end{align*}
and the first inequality holds.

Denote $\kappa_s=\exp(C s)$  and  $\hat{\varphi}_s=\varphi(X^x_s)-\varphi(0)$ for $\varphi=b,h,\sigma$. Then it follows from the $G$-It\^{o}'s formula that
\begin{align*}
&\kappa_t|X^x_t|^2\tilde{X}_t-|x|^2-C\int^t_0\kappa_s|X^x_s|^2\tilde{X}_sds-\bar{M}_t
\\
&\leq 2\int^t_0\kappa_s\langle X^x_s, b(X^x_s)\rangle \tilde{X}_s ds+
2\int^t_0\kappa_sG(2\langle X^x_s, h(X^x_s)\rangle+ |\sigma(X^x_s)|^2)\tilde{X}_s ds+2\int^t_0\kappa_s\langle X^x_s, \sigma(X^x_s)\rangle b_s\tilde{X}_sds
\\
&\ \ \ \ \ \ \
+4\int^t_0\kappa_sG(\langle X^x_s,\sigma(X^x_s)\rangle d_s)\tilde{X}_sds,\\
&\leq 2\int^t_0\kappa_s\langle X^x_s, \hat{b}_s\rangle  \tilde{X}_sds+2\int^t_0\kappa_sG(|\hat{\sigma}_s|^2+2\langle X^x_s, \hat{h}_s\rangle) \tilde{X}_sds+
2\int^t_0\kappa_s\langle X^x_s, \hat{\sigma}_s\rangle b_s\tilde{X}_sds\\
&\ \ \ \ \ \ \ \ +4\int^t_0\kappa_sG(\langle X^x_s,\hat{\sigma}_s\rangle d_s)\tilde{X}_sds+\Lambda^3_t,
\end{align*}
with \begin{align*}
\bar{M}_t=&2\int^t_0\kappa_s\langle X^x_s,\sigma(X^x_s)\rangle \tilde{X}_sd B_s+\int^t_0\kappa_s|X^x_s|^2d_s\tilde{X}_sd B_s
 +\int^t_0\kappa_s|X^x_s|^2b_s\tilde{X}_sd\tilde{B}_s,\\
\Lambda^3_t=&2\int^t_0\kappa_s\langle X^x_s, b(0)\rangle  \tilde{X}_sds
 +2\int^t_0\kappa_sG(2\langle \sigma(0),\hat{\sigma}_s\rangle+\sigma^2(0)
+2\langle X^x_s, h(0)\rangle) \tilde{X}_sds+2\int^t_0\kappa_s\langle X^x_s, \sigma(0)\rangle b_s\tilde{X}_sds
\\ & \ \ \ \ \ \ \ \ \ +4\int^t_0\kappa_sG(\langle X^x_s,\sigma(0)\rangle d_s)\tilde{X}_sds,
\end{align*} where we have used that $\varphi_s=\hat{\varphi}_s+\varphi(0)$ and the sublinearity of $G$ in the last inequality,

In spirit of  assumption (B4), we get that
\begin{align}\label{mw231}
\kappa_t|X^x_t|^2\tilde{X}_t
&\leq |x|^2+\bar{M}_t-C\int^t_0\kappa_s|X^x_s|^2\tilde{X}_sds+\Lambda_t^3.
\end{align}
Recalling $ab\leq \frac{ca^2}{2}+\frac{b^2}{2c}$ for each $c>0$. Then we can find a constant $\tilde{C}$ depending only on $\eta,G,\alpha_1,\alpha_2$ and $b(0),h(0),\sigma(0)$,
so that
\[
\Lambda_t^3\leq C\int^t_0\kappa_s|X^x_s|^2\tilde{X}_sds+\tilde{C}\int^t_0\kappa_s\tilde{X}_sds.
\]
Consequently,  taking expectation on both sides of equation \eqref{mw231}, we derive that
\[
\hat{\mathbb{E}}^{\tilde{G}}[\exp(C t)|X^x_t|^2\tilde{X}_t]\leq |x|^2+\tilde{C}\exp(Ct).
\]
Thus, it follows  that
\[
\hat{\mathbb{E}}^{\tilde{G}}[|X^x_t|\tilde{X}_t]\leq \hat{\mathbb{E}}^{\tilde{G}}[|X^x_t|^2\tilde{X}_t]^{\frac{1}{2}}\leq \sqrt{\tilde{C}}+|x|,
\]
which completes the proof.
\end{proof}

Under assumptions {(B1)-(B5)},  we define%
\[
u(x):=Y_{0}^{x},\ \ x\in\mathbb{R}^n.
\]

\begin{lemma}\label{HM8}
 $u$ is a bounded continuous function. Moreover, there exists some constant $M$ depending only on
$L,\eta,\alpha_1,\alpha_2$ and $G$ such that
\[
|u(x)-u(x^{\prime})|\leq M|x-x^{\prime}|.
\]
\end{lemma}
\begin{proof}
Without loss of generality, assume $d=1$.
By Theorem \ref{HM3}, $Y^x_t$ is bounded by $\frac{\alpha}{\mu}$. In particular, $|u(x)|\leq \frac{\alpha}{\mu}.$
Denote $\tilde{Y} =Y^x-Y^{x^{\prime}}$, $\tilde{Z}=Z^x-Z^{x^{\prime}}$.
Using the same kind of linearization as in Theorem \ref{HM3}, we get
\[
\tilde{Y}_{t}+K_{t}^{x^{\prime}}=\tilde{Y}_{T}+K_{T}^{x^{\prime}}
+\int_{t}^{T}\tilde{f}_{s}ds+\int%
_{t}^{T}\tilde{g}_{s}d\langle B\rangle_{s}-\int_{t}^{T}\tilde{Z}_{s}dB_{s}%
-(K_{T}^{x}-K_{t}^{x}),
\]
where $\tilde{f}_{s}=f(X^x_s,Y_{s}^{x},Z_{s}%
^{x})-f(X^{x^{\prime}}_s,Y_{s}^{x^{\prime}},Z_{s}%
^{x^{\prime}})$, $\tilde{g}_{s}=g(X^x_s,Y_{s}^{x},Z_{s}%
^{x})-g(X^{x^{\prime}}_s,Y_{s}^{x^{\prime}},Z_{s}%
^{x^{\prime}})$.
Consequently, by Lemma \ref{HW4}, we get for each $\epsilon>0$
\[
\tilde{f}_{s}=a^{\epsilon}_{s}\tilde{Y}_{s}+b^{\epsilon}_{s}\tilde{Z}%
_{s}+m^{\epsilon}_{s}+m_s,\ \hat{g}_{s}=c^{\epsilon}_{s}\tilde{Y}%
_{s}+d^{\epsilon}_{s}\tilde{Z}_{s}+n^{\epsilon}_{s}+n_s,
\]
where $m_{s}=f(X^x_s,Y_{s}^{x^{\prime}},Z_{s}%
^{x^{\prime}})-f(X^{x^{\prime}}_s,Y_{s}^{x^{\prime}},Z_{s}%
^{x^{\prime}})$,  $n_{s}=g(X^x_s,Y_{s}^{x^{\prime}},Z_{s}%
^{x^{\prime}})-g(X^{x^{\prime}}_s,Y_{s}^{x^{\prime}},Z_{s}%
^{x^{\prime}})$ and $|m^{\epsilon}_{s}|\leq 2(L+\alpha_2)\epsilon$, $|n^{\epsilon}_{s}|\leq 2(L+\alpha_2)\epsilon$,
$a^{\epsilon}_{s}+2G(c^{\epsilon}_{s})\leq -\mu$.
Recalling Lemma \ref{the5.2}, we obtain that
\begin{align}\label{HM5}
\tilde{Y}_{0}&  \leq \mathbb{\hat{E}}^{\tilde{G}}%
[X^{\epsilon}_{T}\tilde{Y}_{T}+\int_{0}^{T}m_{s}X^{\epsilon}_{s}d{s}+\int_{0}^{T}n_{s}X^{\epsilon}_{s}d\langle B\rangle_{s}]
+\mathbb{\hat{E}}^{\tilde{G}}%
[\int_{0}^{T}m^{\epsilon}_{s}X^{\epsilon}_{s}d{s}+\int_{0}^{T}n^{\epsilon}_{s}X^{\epsilon}_{s}d\langle B\rangle_{s}],
\nonumber\\ & \leq   \mathbb{\hat{E}}^{\tilde{G}}%
[X^{\epsilon}_{T}\tilde{Y}_{T}+(1+\bar{\sigma}^2)L\int_{0}^{T}|X^x_s-X^{x^{\prime}}_s|X^{\epsilon}_{s}d{s}]
+\mathbb{\hat{E}}^{\tilde{G}}%
[\int_{0}^{T}m^{\epsilon}_{s}X^{\epsilon}_{s}d{s}+\int_{0}^{T}n^{\epsilon}_{s}X^{\epsilon}_{s}d\langle B\rangle_{s}],\ \ q.s.,
\end{align}
where $\{X^{\epsilon}_{t}\}_{t\in\lbrack0,T]}$ is given by
\begin{equation*}
X^{\epsilon}_{t}=\exp(\int_{0}^{t}(a^{\epsilon}_{s}-b^{\epsilon}_{s}d^{\epsilon}_{s})ds+\int_{0}^{t}c^{\epsilon}_{s}d\langle
B\rangle_{s})\mathcal{E}_{t}^{B}\mathcal{E}_{t}^{\tilde{B}}.
\end{equation*}
Here $\mathcal{E}_{t}^{B}=\exp(\int^t_0d^{\epsilon}_{s}dB_s-\frac{1}{2}\int^t_0|d^{\epsilon}_{s}|^2d\langle B\rangle_s)$ and
$\mathcal{E}_{t}^{\tilde{B}}=\exp(\int^t_0b^{\epsilon}_{s}d\tilde{B}_s-\frac{1}{2}\int^t_0|b^{\epsilon}_{s}|^2d\langle \tilde{B}\rangle_s)$.
Thus
\[
|X^x_t-X^{x^{\prime}}_t||X^{\epsilon}_t|\leq \exp(-\mu t)|X^x_t-X^{x^{\prime}}_t|\tilde{X}^{\epsilon}_t,
\]
where \[
\tilde{X}^{\epsilon}_{t}=1+\int_{0}^{t}d^{\epsilon}_{s}\tilde{X}^{\epsilon}_{s}dB_{s}+\int_{0}%
^{t}b^{\epsilon}_{s}\tilde{X}^{\epsilon}_{s}d\tilde{B}_{s}.
\]
From Lemma \ref{HW2}, we conclude that
\[
\mathbb{\hat{E}}^{\tilde{G}}[|X^x_t-X^{x^{\prime}}_t|X^{\epsilon}_t]\leq \exp(-\mu t-\eta t+(1+\bar{\sigma}^2)\alpha_1\alpha_2t)|x-x^{\prime}|.
\]
Thus by equation \eqref{HM5} and  sending $\epsilon\rightarrow 0$, we deduce that
\[
u(x)-u(x^{\prime})\leq \exp(-\mu T)\frac{\alpha}{\mu}+\frac{(1+\bar{\sigma}^2)L}{\mu+\eta-(1+\bar{\sigma}^2)\alpha_1\alpha_2}|x-x^{\prime}|,
\]
Letting $T\rightarrow\infty$, we obtain $u(x)-u(x^{\prime})\leq \frac{(1+\bar{\sigma}^2)L}{\eta-(1+\bar{\sigma}^2)\alpha_1\alpha_2}|x-x^{\prime}|$.
In a similar  way, we  also have
$u(x^{\prime})-u(x)\leq \frac{(1+\bar{\sigma}^2)L}{\eta-(1+\bar{\sigma}^2)\alpha_1\alpha_2}|x-x^{\prime}|$,
which is the desired result.
\end{proof}

Now we shall present the main results of this section.
\begin{lemma}\label{mw12}
For each $(t,x)\in[0,\infty)\times\mathbb{R}^n$,  we have
 $Y_{t}^{x}=u(X^x_t)$.
 \end{lemma}
 The proof will be given in the appendix.

\begin{theorem}\label{HM9}
 $u(x)$ is the unique bounded continuous viscosity solution of the
following PDE:%
\begin{equation}
G(H(D_{x}^{2}u,D_{x}u,u,x))+\langle
b(x),D_{x}u\rangle+f(x,u,D_{x}u\sigma(x))=0,
\label{feynman}%
\end{equation}
where
\begin{align*}
H_{ij}(D_{x}^{2}u,D_{x}u,u,x)=   \langle D_{x}^{2}u\sigma_{i}%
(x),\sigma_{j}(x)\rangle+2\langle D_{x}u,h_{ij}(x)\rangle +2g_{ij}(x,u,D_{x}u
\sigma(x)).
\end{align*}
\end{theorem}
\begin{proof}
The uniqueness of viscosity solution of equation (\ref{feynman}) will be given in appendix.
Applying Lemma \ref{mw12}, we obtain for each $\delta>0$,
\begin{align*}
u(x)=  &  u(X_{\delta}^{x})+\int_{0}^{\delta}%
f(X_{r}^{x},u(X_{r}^{x}),Z_{r}^{x})dr +\int_{0}^{\delta}g_{ij}(X_{r}^{x},u(X_{r}^{x}),Z_{r}^{x})d\langle
B^{i},B^{j}\rangle_{r}\\
& -\int_{0}^{\delta}Z_{r}^{x}dB_{r}-K_{\delta
}^{x},
\end{align*}
which can be seen as a $G$-BSDE in finite horizon $[0,\delta]$.
Then we can prove that $u$ is a viscosity solution of equation \eqref{feynman} by a similar way as  the proof of Theorem 4.5 in  \cite{HJPS1}. Indeed,  it is easier than  the one  of\cite{HJPS1} in our case, since there is no
time variable (see also \cite{PE, RM} in the linear case). The proof is complete.
\end{proof}

In the next theorem, we shall discuss the sign of the solution of equation \eqref{feynman}.
\begin{theorem}
Suppose moreover that  $-f(X^x_s,0,0)+2G(-g_{ij}(X^x_s,0,0))\leq 0$ for each $s>0$.
Then $ u(x)\geq 0$.
\end{theorem}
\begin{proof}
It follows from Comparison Theorem \ref{HM7} that $\forall t\geq 0$, $Y^x_t\geq 0$.
In particular, for $t=0$, we deduce
that $u(x)\geq 0.$
\end{proof}
\begin{remark}{\upshape
In order to state the $G$-EBSDE, we establish the fully nonlinear Feynman-Kac formula for elliptic PDEs under stronger assumptions (B1)-(B5).
However, the assumptions (B4) and (B5) can be removed as the linear case through a uniform continuity argument. These more technical
details are left to future work.}
\end{remark}
\section{Ergodic BSDEs driven by $G$-Brownian motion}
In this section, we shall study the following type of (Markovian) ergodic BSDEs driven by
$G$-Brownian motion under assumptions (B1), (B2), (B4) and (B5) ($\mu=0$):
\begin{align}
Y_{s}^{x}   =Y^x_T+\int_{s}^{T}[f(X_{r}^{x},Z_{r}^{x})+\gamma^1\lambda]dr+\int_{s}^{T}[g_{ij}(X_{r}^{x}%
,Z_{r}^{x})+\gamma^2_{ij}\lambda]d\langle B^i,B^j\rangle_{r} -\int_{s}^{T}Z_{r}^{x}dB_{r}-(K_{T}^{x}-K_{s}^{x}),
\label{AppHM}%
\end{align}
where $\gamma^1$ is a fixed constant   and $\gamma^2 $ is a  given $d\times d$ symmetric matrix satisfied $\gamma^1+2G(\gamma^2)<0$ as in introduction.

As in \cite{BH}, we start by considering an infinite horizon equation with strictly monotonic drift, namely
for each $\epsilon>0$, the $G$-BSDEs:
\begin{align}
Y_{s}^{x,\epsilon}  &  =Y^{x,\epsilon}_T+\int_{s}^{T}[f(X_{r}^{x},Z_{r}^{x,\epsilon})+\gamma^1\epsilon Y_{r}^{x,\epsilon}]dr+\int_{s}^{T}[g_{ij}(X_{r}^{x}%
,Z_{r}^{x,\epsilon})+\gamma^2_{ij}\epsilon Y_{r}^{x,\epsilon}]d\langle B^i,B^j\rangle_{r}\nonumber\\
&  -\int_{s}^{T}Z_{r}^{x,\epsilon}dB_{r}-(K_{T}^{x,\epsilon}-K_{s}^{x,\epsilon}).
\label{AppG2}%
\end{align}
From Theorem \ref{HM3}, we immediately have
\begin{lemma}\label{HM4}
The  $G$-BSDE \eqref{AppG2} has a unique solution $(Y^{x,\epsilon}, Z^{x,\epsilon}, K^{x,\epsilon})$ belonging to $\mathfrak{S}_{G}^{2}(0,\infty)$ such that $Y^{x,\epsilon}$ is a bounded process. Furthermore, $|Y^{\epsilon,x}_t|\leq \frac{\alpha}{-(\gamma^1+2G(\gamma^2))\epsilon}.$
\end{lemma}

Then denote $
v^{\epsilon}(x):=Y^{x,\epsilon}_0.
$
Then by Lemma \ref{HM8}, we have
\begin{lemma}There exists some constant $M>0$ independent of $\epsilon$ such that
\[
|v^{\epsilon}(x)-v^{\epsilon}(x^{\prime})|\leq M|x-x^{\prime}|.
\]
\end{lemma}

Denote $\bar{v}^{\epsilon}(x)={v}^{\epsilon}(x)-{v}^{\epsilon}(0)$. Then $|\bar{v}^{\epsilon}(x)|\leq M|x|$ and $\epsilon{v}^{\epsilon}(0)\leq \frac{\alpha}{-\gamma^1-2G(\gamma^2)}.$
Note that $\bar{v}^{\epsilon}(x)$ is a $M$-Lipschitz function for each $\epsilon$.
Thus by a diagonal procedure we can construct  a
sequence $\epsilon_n\downarrow 0$ such that
$\bar{v}^{\epsilon_n}(x)\rightarrow v(x)$ for all $x\in\mathbb{R}^n$ and $\epsilon_n{v}^{\epsilon_n}(0)\rightarrow \lambda$, where $\lambda$ is a real number.

\begin{theorem}\label{my11}
Suppose assumptions \emph{(B1), (B2), (B4) } and \emph{(B5)} hold. Then for each $x$, the
$G$-EBSDE \eqref{AppHM} has a solution $(Y^x,Z^x,K^x,\lambda)\in\mathfrak{S}_{G}^{2}(0,\infty)\times\mathbb{R}$  such that $|Y^x_s|\leq M|X^x_s|$.
\end{theorem}
\begin{proof}
Denote $Y^x_t:=v(X^x_t)$ and $\bar{Y}^{x,\epsilon_n}_t:=Y^{x,\epsilon_n}_t-{v}^{\epsilon}(0)=\bar{v}^{\epsilon_n}(X^x_t)$ for each $n$. Then  we have for each $T>0$,
\begin{align*}
\bar{Y}_{s}^{x,\epsilon_n}  &  =\bar{Y}^{x,\epsilon_n}_T+\int_{s}^{T}[f(X_{r}^{x},Z_{r}^{x,\epsilon_n})+\gamma^1\epsilon_n \bar{Y}_{r}^{x,\epsilon_n}+\gamma^1\epsilon_n{v}^{\epsilon_n}(0)]dr-\int_{s}^{T}Z_{r}^{x,\epsilon_n}dB_{r}-(K_{T}^{x,\epsilon_n}-K_{s}^{x,\epsilon_n})\nonumber\\&
+\int_{s}^{T}[g_{ij}(X_{r}^{x}%
,Z_{r}^{x,\epsilon_n})+\gamma^2_{ij}\epsilon_n \bar{Y}_{r}^{x,\epsilon_n}+\gamma^2_{ij}\epsilon_n{v}^{\epsilon_n}(0)]d\langle B^i,B^j\rangle_{r}.
\end{align*}
Note that $\bar{v}^{\epsilon_n}(x)$ converges to $ v(x)$ uniformly on any compact subset of $\mathbb{R}^n$. Then for each $N>0$, we get that
\begin{align*}
&\lim\limits_{n\rightarrow\infty}\mathbb{\hat{E}}[\sup\limits_{t\in[0,T]}|\bar{Y}^{x,\epsilon_n}_t-Y^x_t|^2]\\
&\leq
\lim\limits_{n\rightarrow\infty}\mathbb{\hat{E}}[\sup\limits_{t\in[0,T]}|\bar{v}^{\epsilon_n}(X^x_t)-v(X^x_t)|^2\mathbf{1}_{|X^x_t|\leq N}]
+\lim\limits_{n\rightarrow\infty}\mathbb{\hat{E}}[\sup\limits_{t\in[0,T]}|\bar{v}^{\epsilon_n}(X^x_t)-v(X^x_t)|^2\mathbf{1}_{|X^x_t|\geq N}]
\\
& \leq \lim\limits_{n\rightarrow\infty}\sup\limits_{|x|\leq N}|\bar{v}^{\epsilon_n}(x)-v(x)|^2+2M^2\frac{\mathbb{\hat{E}}[\sup\limits_{t\in[0,T]}|X_t^x|^3]}{N}=2M^2\frac{\mathbb{\hat{E}}[\sup\limits_{t\in[0,T]}|X_t^x|^3]}{N}.
\end{align*}
Recalling Proposition 4.1 in \cite{HJPS1} and letting $N\rightarrow\infty$, we conclude that
\[
\lim\limits_{n\rightarrow\infty}\mathbb{\hat{E}}[\sup\limits_{t\in[0,T]}|\bar{Y}^{x,\epsilon_n}_t-Y^x_t|^2]=0.
\]
Applying Theorem \ref{pro3.5} (note that  $C$ can be taken as a generic constant independent of $n$, since $\epsilon_n$ is uniformly bounded), there exist two processes $Z^{x,T}_t$ and $K^{x,T}_t$ such that
\[
\lim\limits_{n\rightarrow\infty}\mathbb{\hat{E}}[\int^T_0|Z^{x,\epsilon_n}_t-Z^{x,T}_t|^2dt]=0, \ \ \lim\limits_{n\rightarrow\infty}\mathbb{\hat{E}}[|K^{x,\epsilon_n}_t-K^{x,T}_t|^2]=0.
\]
Moreover, $(Y^x,Z^{x,T}_t, K^{x,T}_t )$ satisfies the following equation
\begin{align*}
Y_{s}^{x}   =&Y^x_T+\int_{s}^{T}[f(X_{r}^{x},Z_{r}^{x,T})+\gamma^1\lambda]dr+\int_{s}^{T}[g_{ij}(X_{r}^{x}%
,Z_{r}^{x,T})+\gamma^2_{ij}\lambda]d\langle B^i,B^j\rangle_{r}\\
& \ \ \ \ \  -\int_{s}^{T}Z_{r}^{x,T}dB_{r}-(K_{T}^{x,T}-K_{s}^{x,T}).
\end{align*}
By the uniqueness of solution to $G$-BSDE, it is obvious $(Z_{r}^{x,T},K_{r}^{x,T})=(Z_{r}^{x,S},K_{r}^{x,S})$ for $S>T$.
Set $(Z^x_t,K^x_t)=(Z^{x,T}_t,K^{x,T}_t)$ for some $T\geq t$. Then $(Y^x,Z^x,K^x,\lambda)$
satisfies equation \eqref{AppHM}. The proof is complete.
\end{proof}

 Based on the $G$-EBSDE,
we  could show the  following fully nonlinear ergodic elliptic PDE has a viscosity pair solution:
\begin{align} \label{HM11}
&G(H(D_{x}^{2}{v},D_{x}{v},\lambda,x))+\langle
b(x),D_{x}{v}\rangle+f(x,D_{x}{v}\sigma(x))+\gamma^1\lambda=0,
\end{align}%
where
\begin{align*}
H_{ij}(D_{x}^{2}{v},D_{x}{v},\lambda,x)=  &  \langle D_{x}^{2}{v}\sigma_{i}%
(x),\sigma_{j}(x)\rangle+2\langle D_{x}{v},h_{ij}(x)\rangle
 +2g_{ij}(x,D_{x}{v}
\sigma(x))+2\gamma_{ij}^2\lambda.
\end{align*}
\begin{definition}\label{myq19}\emph{(i)} A viscosity pair subsolution (resp. suppersolution) of \eqref{HM11} is a
 pair $(u,\lambda)$ with   a real number $\lambda$  and a upper  (resp. lower)
semicontinuous function $u$, such that for all $x\in\mathbb{R}^n$ and $\varphi\in C^2(\mathbb{R}^n)$ satisfying
$\varphi(y)-u(y)\geq (\text{resp.} \leq)\varphi(x)-u(x)$ for each $y\in\mathbb{R}^n$, we have
\[
G(H(D_{x}^{2}{\varphi},D_{x}{\varphi},\lambda,x))+\langle
b(x),D_{x}{\varphi}\rangle+f(x,D_{x}\varphi\sigma
(x))+\gamma^1\lambda\geq (\text{resp.} \leq) 0.
\]
\emph{(ii)}  A viscosity pair solution of \eqref{HM11} is a
 pair $(u,\lambda)$ with   a real number  $\lambda$ and a continuous
function $u$, such that it is simultaneously a viscosity  pair subsolution and a  viscosity  pair suppersolution.
\end{definition}

\begin{remark}
{\upshape
Note that the equation \eqref{HM11} is a fully nonlinear elliptic PDE in $(D^2_xv, \lambda)$,
which is different from the previous works (see \cite{BF,CF,Fuj,H} and  the references
therein).
}
\end{remark}
\begin{theorem}
Assume assumptions \emph{(B1), (B2), (B4) } and \emph{(B5)} hold. Then ergodic PDE \eqref{HM11} has a  viscosity pair solution $(v,\lambda)$.
\end{theorem}
\begin{proof}
Consider $(v,\lambda)$ given in Theorem \ref{my11}. For each $x$, denote $Y^x_s:=v(X^x_s)$. Then we have for each $T>0$,
\begin{align}\label{HW3}
v(X^x_s) &  =v(X^x_T)+\int_{s}^{T}[f(X_{r}^{x},Z_{r}^{x})+\gamma^1\lambda]dr+\int_{s}^{T}[g_{ij}(X_{r}^{x}%
,Z_{r}^{x})+\gamma_{ij}^2\lambda]d\langle B^i,B^j\rangle_{r}\nonumber\\
& \ \ -\int_{s}^{T}Z_{r}^{x}dB_{r}-(K_{T}^{x}-K^{x}_s).
\end{align}
By the nonlinear Feynman-Kac formula in \cite{HJPS1}, we obtain $v$ is the unique viscosity solution to the following parabolic PDE:
\begin{align*}
\begin{cases}
& \partial_t\phi(t,x)+G(H(D_{x}^{2}\phi,D_{x}\phi,\lambda,x))+
\langle b(x), D_{x}\phi\rangle+f(x,\langle\sigma_{1}(x),D_{x}{\phi}\rangle,\ldots,\langle\sigma
_{d}(x),D_{x}{\phi}\rangle)+\gamma_1\lambda=0,
\\& \phi(T,x)=v(x).
\end{cases}
\end{align*}
Then by the Definition \ref{myq19}, one  can easily check that $(v,\lambda)$ is a viscosity pair solution of \eqref{HM11}.
\end{proof}
\begin{remark}{\upshape
Note that the nonlinear expectation theory is a useful tool to deal with nonlinear  ergodic  problems and we intend to carry  over these ideas to more general cases, for example, HJB equations and nonlinear ``invariant measures'' (see \cite{HL1}).
}\end{remark}

It is obvious the solution to $G$-EBSDE \eqref{AppHM} is not unique.
 Indeed the equation is invariant
with respect to addition of a constant to $Y$.   However we have a uniqueness result for $\lambda$ under some additional condition.

\begin{theorem}
If for some $x\in\mathbb{R}^n$, $(Y^{\prime,x},Z^{\prime,x}, K^{\prime,x},\lambda^{\prime})\in\mathfrak{S}_{G}^{2}(0,\infty)\times\mathbb{R}$ verifies equation \eqref{AppHM}. Moreover, there exists some constant $c^x > 0$ such that\[
|Y^{\prime,x}_s|\leq c^x(1+|X^x_s|).
\]
Then $\lambda^{\prime}=\lambda.$
\end{theorem}
\begin{proof}
Without loss of generality, assume $d=1$ and $\lambda\geq \lambda^{\prime}$.
Set $(\hat{Y},\hat{Z},\hat{\lambda})=(Y^{x}-Y^{\prime,x},Z^{x}-Z^{\prime,x},\lambda-\lambda^{\prime})$.
Then we have for each $T$ and $\epsilon$,
\[
\hat{Y}_{t}+K_{t}^{\prime,x}=\hat{Y}_T+K_{T}^{\prime,x}+\int_{t}^{T}[\hat{f}_{s}+\gamma^1\hat{\lambda}]ds+\int%
_{t}^{T}[\hat{g}_{s}+\gamma^2\hat{\lambda}]d\langle B\rangle_{s}-\int_{t}^{T}\hat{Z}_{s}dB_{s}%
-(K_{T}^{x}-K_{t}^{x}),
\]
where $\hat{f}_{s}=f(X^x_s,Z_{s}
^{x})-f(X^x_s,Z_{s}
^{x,\prime})=b_{s}^{\epsilon}\hat{Z}%
_{s}+m^{\epsilon}_s$, $ \hat{g}_{s}=g(X^x_s,Z_{s}
^{x})-g(X^x_s,Z_{s}
^{x,\prime})=d_{s}^{\epsilon}\hat{Z}
_{s}+n^{\epsilon}_s$, $|m^{\epsilon}_s|\leq 2\alpha_2\epsilon$ and $|n^{\epsilon}_s|\leq 2\alpha_2\epsilon$.
By a similar analysis as in Theorem \ref{HM3}, we obtain
\begin{align*}
\hat{Y}_{0}&  \leq \mathbb{\hat{E}}^{\tilde{G}}%
[X_{T}^{\epsilon}\hat{Y}_{T}+\int_{0}^{T}\gamma^1\hat{\lambda}X_{s}^{\epsilon}d{s}+\int_{0}^{T}\gamma^2\hat{\lambda}X_{s}^{\epsilon}d\langle B\rangle_{s}]+
\mathbb{\hat{E}}^{\tilde{G}}%
[\int_{0}^{T}m^{\epsilon}_{s}X^{\epsilon}_{s}d{s}+\int_{0}^{T}n^{\epsilon}_{s}X^{\epsilon}_{s}d\langle B\rangle_{s}]\\
& \leq \mathbb{\hat{E}}^{\tilde{G}}%
[X^{\epsilon}_{T}\hat{Y}_{T}]+\mathbb{\hat{E}}^{\tilde{G}}[\int_{0}^{T}(\gamma^1\hat{\lambda}+2G(\gamma^2\hat{\lambda}))
X^{\epsilon}_{s}d{s}]+2(1+\bar{\sigma}^2)\alpha_2T\epsilon,
\end{align*}
where $\{X^{\epsilon}_{t}\}_{t\in\lbrack0,T]}$ is the solution of the
following $\tilde{G}$-SDE:%
\[
X^{\epsilon}_{t}=1+\int_{0}^{t}d_{s}^{\epsilon}X^{\epsilon}_{s}dB_{s}+\int_{0}%
^{t}b^{\epsilon}_{s}X^{\epsilon}_{s}d\tilde{B}_{s}.
\]
By the $G$-It\^{o}'s formula, we derive that $\mathbb{\hat{E}}^{\tilde{G}}[\int_{0}^{T}(\gamma^1\hat{\lambda}+2G(\gamma^2\hat{\lambda}))
X^{\epsilon}_{s}d{s}]=\mathbb{\hat{E}}^{\tilde{G}}[(\gamma^1+2G(\gamma^2))\hat{\lambda}X^{\epsilon}_{T}T]=
(\gamma^1+2G(\gamma^2))\hat{\lambda}T$.
Consequently,
\begin{align*}
\hat{Y}_{0}\leq \mathbb{\hat{E}}^{\tilde{G}}%
[X^{\epsilon}_{T}\hat{Y}_{T}]+(\gamma^1+2G(\gamma^2))\hat{\lambda}T+2(1+\bar{\sigma}^2)\alpha_2T\epsilon.
\end{align*}
 Recalling Lemma \ref{HW2}, there exists some constant $C$ such that
\[
\mathbb{\hat{E}}^{\tilde{G}}%
[|X^{\epsilon}_{T}\hat{Y}_{T}|]\leq C(1+|x|).
\]
Thus letting $\epsilon\rightarrow 0$,  we can find some constant $C$ depending on $G$ and $c_x, \bar{C}$ such that for each $T$,
\[
\lambda-{\lambda}^{\prime}\leq \frac{C}{T}(1+|x|).
\]
Consequently, letting $ T \rightarrow\infty$  yields  that $\lambda\leq {\lambda}^{\prime}$,
which concludes the result.
\end{proof}

\section{Applications}
\subsection{Large time behaviour of solutions to fully nonlinear PDEs}
In this section, we shall apply the $G$-EBSDEs to obtain the large time behaviour of solutions to fully nonlinear PDEs where the diffusion term may be degenerate.
 Let us consider the following $G$-EBSDE:
\begin{align}
Y_{s}^{x}   =Y^x_T+\int_{s}^{T}[f(X_{r}^{x},Z_{r}^{x})-\lambda]dr+\int_{s}^{T}g_{ij}(X_{r}^{x}%
,Z_{r}^{x})d\langle B^i,B^j\rangle_{r} -\int_{s}^{T}Z_{r}^{x}dB_{r}-(K_{T}^{x}-K_{s}^{x}),
\label{AppHMW}%
\end{align}
and the fully nonlinear  ergodic PDE:
\begin{align} \label{HW11}
&G(H(D_{x}^{2}{v},D_{x}{v},x))+\langle
b(x),D_{x}{v}\rangle+f(x,D_{x}{v}\sigma
(x))=\lambda,
\end{align}%
where
\begin{align*}
H_{ij}(D_{x}^{2}{v},D_{x}{v},x)=  &  \langle D_{x}^{2}{v}\sigma_{i}%
(x),\sigma_{j}(x)\rangle+2\langle D_{x}{v},h_{ij}(x)\rangle+2g_{ij}(x,D_{x}{v}
\sigma(x)).
\end{align*}
From the  section 5,  the $G$-EBSDE \eqref{AppHMW} and  the  fully
nonlinear ergodic PDE \eqref{HW11} both have  solutions. Moreover, the constant $\lambda$ in the ergodic equation \eqref{HW11} is
unique.

For each Lipschitz function $\varphi:\mathbb{R}^n\rightarrow\mathbb{R}$, consider the following fully nonlinear parabolic PDE:
\begin{align}
\begin{cases}
& \partial_tu(t,x)-G(H(D_{x}^{2}u,D_{x}u,x))-
\langle b(x),D_{x}u\rangle -f(x,D_{x}u\sigma(x))=0,
\\& u(0,x)=\varphi(x).
\end{cases}
\end{align}
Denote $u^T(t,x):=u(T-t,x)$ for each $T>0$. Then $u^T(t,x)$ is the unique viscosity solution of PDE:
\begin{align*}
\begin{cases}
& \partial_tu^T(t,x)+G(H(D_{x}^{2}u^T,D_{x}u^T,x))+
\langle b(x),D_{x}u^T\rangle +f(x,D_{x}{u^T}\sigma(x))=0,
\\& u^T(T,x)=\varphi(x).
\end{cases}
\end{align*}
\begin{theorem}
Under assumptions \emph{(B1), (B2), (B4)} and \emph{(B5)}, there exists a constant $C$ such that, for each $T>0$,
\[
|\frac{u(T,x)}{T}-\lambda|\leq \frac{C(1+|x|)}{T}.
\]
In particular,
\[
\lim\limits_{T\rightarrow\infty}\frac{u(T,x)}{T}=\lambda.
\]
\end{theorem}

\begin{proof}
For  convenience, assume $d=1$.
Recalling nonlinear Feynman-Kac formula in \cite{HJPS1}, we obtain for each $s\in[0,T]$,
\begin{align*}
u^T(s,X^x_s) &  =\varphi(X^x_T)+\int_{s}^{T}f(X_{r}^{x},Z_{r}^{x,T})dr+\int_{s}^{T}g(X_{r}^{x}%
,Z_{r}^{x,T})d\langle B\rangle_{r}
 -\int_{s}^{T}Z_{r}^{x,T}dB_{r}-(K_{T}^{x,T}-K^{x,T}_s).
\end{align*}
From equation \eqref{HW3}, we conclude
\[
\hat{Y}_{t}+K_{t}^{x}=\hat{Y}_T+K_{T}^{x}+\int_{t}^{T}[\hat{f}_{s}+{\lambda}]ds+\int%
_{t}^{T}\hat{g}_{s}d\langle B\rangle_{s}-\int_{t}^{T}\hat{Z}_{s}dB_{s}%
-(K_{T}^{x,T}-K_{t}^{x,T}),
\]
where $(\hat{Y},\hat{Z})=(u^T(\cdot,X^x_{\cdot})-Y^{x},Z^{x,T}-Z^{x})$, $\hat{f}_{s}=f(X^x_s,Z_{s}
^{x,T})-f(X^x_s,Z_{s}
^{x})=b_{s}^{\epsilon}\hat{Z}%
_{s}+m^{\epsilon}_s$ and $\hat{g}_{s}=g(X^x_s,Z_{s}
^{x,T})-g(X^x_s,Z_{s}
^{x})=d_{s}^{\epsilon}\hat{Z}%
_{s}+n_s^{\epsilon}$ for each $\epsilon>0$. Here $|m^{\epsilon}_s|\leq 2\alpha_2\epsilon $ and  $|n^{\epsilon}_s|\leq 2\alpha_2\epsilon $.
By a standard argument, we derive that, in the extended space,
\begin{align*}
\hat{Y}_{0}&  \leq \mathbb{\hat{E}}^{\tilde{G}}%
[X^{\epsilon}_T\hat{Y}_{T}+\int_{0}^{T}{\lambda}X^{\epsilon}_{s}d{s}]+2(1+\bar{\sigma}^2)LT\epsilon=\mathbb{\hat{E}}^{\tilde{G}}%
[X^{\epsilon}_{T}\hat{Y}_{T}]+{\lambda}T+2(1+\bar{\sigma}^2)\alpha_2T\epsilon,
\end{align*}
where $\{X^{\epsilon}_{t}\}_{t\in\lbrack0,T]}$ is the solution of the
following $\tilde{G}$-SDE:%
\[
X^{\epsilon}_{t}=1+\int_{0}^{t}d_{s}^{\epsilon}X^{\epsilon}_{s}dB_{s}+\int_{0}%
^{t}b^{\epsilon}_{s}X^{\epsilon}_{s}d\tilde{B}_{s}.
\]
Denote by $C_0$ a constant that depends only on $v$ and $\varphi$, which is
allowed to change from line to line. Consequently, we have
\[
u(T,x)-v(x)-\lambda T\leq \mathbb{\hat{E}}^{\tilde{G}}%
[X^{\epsilon}_{T}|\varphi(X^x_T)-v(X^x_T)|]+2(1+\bar{\sigma}^2)\alpha_2T\epsilon\leq C_0\mathbb{\hat{E}}^{\tilde{G}}%
[X^{\epsilon}_{T}(1+|X^x_T|)]+2(1+\bar{\sigma}^2)\alpha_2T\epsilon.
\]
In a similar way, we can also get \[
v(x)+\lambda T-u(T,x)\leq C_0\mathbb{\hat{E}}^{\tilde{G}}%
[X^{\epsilon}_{T}(1+|X^x_T|)]+2(1+\bar{\sigma}^2)\alpha_2T\epsilon.
\]
Sending $\epsilon\rightarrow0$ and recalling Lemma \ref{HW2},  there exists  some constant $C$ depending on $G$ and $M,C_0,L$ such that for each $T$,
\[
|u(T,x)-v(x)-\lambda T|\leq C(1+|x|),
\]
which ends the proof.
\end{proof}

\begin{remark}{\upshape
Suppose $f(x,z)$ and $g(x,z)$ are independent of $z$.
One can easily show that\begin{align*}
v^{\epsilon}(x)=&\lim\limits_{T\rightarrow\infty}\mathbb{\hat{E}}[\exp(-\epsilon T)Y^{x,\epsilon}_T+\int_{0}^T\exp(-\epsilon s)f(X_{s}^{x})ds+\int_{0}^{T}\exp(-\epsilon s)g_{ij}(X_{s}^{x})d\langle B^i,B^j\rangle_{s}]\\
=&\mathbb{\hat{E}}[\int_{0}^{\infty}\exp(-\epsilon s)f(X_{s}^{x})ds+\int_{0}^{\infty}\exp(-\epsilon s)g_{ij}(X_{s}^{x})d\langle B^i,B^j\rangle_{s}].
\end{align*}
Then we obtain
\begin{align*}
&\lim\limits_{T\rightarrow\infty}\frac{1}{T}\mathbb{\hat{E}}[\int_{0}^{T}f(X_{s}^{x})ds+\int_{0}^{T}g_{ij}(X_{s}^{x})d\langle B^i,B^j\rangle_{s}]\\
&=\lim\limits_{\epsilon\rightarrow0}\epsilon\mathbb{\hat{E}}[\int_{0}^{\infty}\exp(-\epsilon s)f(X_{s}^{x})ds+\int_{0}^{\infty}\exp(-\epsilon s)g_{ij}(X_{s}^{x})d\langle B^i,B^j\rangle_{s}]\\
&=\lambda,
\end{align*}
which can be seen as Abelian-Tauberian Theorem under $G$-expectation framework.
}
\end{remark}
\begin{remark}{
\upshape
Note that  Fujita, Ishii and Loreti \cite{Fuj} (see also \cite{N1} for further research) studied the asymptotics of semi-linear PDE through analytic approaches under nondegeneracy assumption on the diffusion term. }
\end{remark}
\begin{remark}{
\upshape
Remark that from the results of
 Chapter V in Peng \cite{P10}, we can extend our result to the  case that the sublinear function $G$ is degenerate and $f,g$  is independent of $z$.
 In a different setting, Cosso, Fuhrman and Pham \cite{CF}  used a tricky BSDE approach to obtain
the large time behavior of solutions to general HJB equations, where $f$ does not contain $z$. An interesting question is how to obtain the rate of convergence.
}
\end{remark}
\subsection{Optimal ergodic control under model uncertainty}
The objective of this section is to study optimal ergodic control problems under the model uncertainty.
Let $U$ be a closed subset of $\mathbb{R}^n$. We define a control $u_s\in M^{2}_G(0,\infty)$ as a
$U$-valued process. Let $R:  U\mapsto\mathbb{R}^d $  and $\kappa: \mathbb{R}^n\times U\mapsto\mathbb{R} $ be two bounded $L$-Lipschitz functions.
Moreover, $|R(u)|\leq \alpha_2$.
For each control $u_s\in M^{2}_G(0,\infty)$,
we introduce the following Girsanov transformation  under $G$-expectation framework, which is given in \cite{HJPS1}.
For each $T>0$ and $\xi\in L_G^{2}(\Omega_T)$, consider the following $G$-BSDE:%
\[
Y_{t}=\xi+\int_{t}^{T}R(u_s)Z_{s}ds-\int_{t}^{T}Z_{s}dB_{s}-(K_{T}-K_{t}),
\]
Then $\mathbb{\tilde{E}}^u%
_{t}[\xi]:=Y_{t}$  is a consistent sublinear expectation and $B^u_t:=B_{t}-\int_{0}^{t}R(u_s)d{s}$ is a $G$-Brownian motion under
$\mathbb{\tilde{E}}^u$.

Under the model uncertainty, the nonlinear ergodic cost corresponding to $u$ and the starting
point $x$ is
\begin{align}\label{o1}
J(x,u)=\limsup\limits_{T\rightarrow\infty}\frac{1}{T}\mathbb{\tilde{E}}^{u}[\int^T_0\kappa(X^x_s,u_s)ds].
\end{align}
Our purpose is to minimize costs $J$  over all controls.
Then define the Hamiltonian
in the usual way
\begin{align}\label{o2}
f(x,z)=\inf\limits_{u}(\kappa(x,u)+R(u)z).
\end{align}

From section 5,  the $G$-EBSDE \eqref{AppHMW} ($g=0$) has a solution $(Y^x,Z^x, K^x,\lambda)$ such that
\[
|Y^x_s|\leq M|X^x_s|.
\]

\begin{theorem}
Suppose  assumptions \emph{(B1), (B2), (B4)} and \emph{(B5)} hold. If for some $x\in\mathbb{R}^n$, $(Y,Z, K,\lambda^{\prime})\in\mathfrak{S}_{G}^{2}(0,\infty)\times\mathbb{R}$ satisfies equation \eqref{AppHMW}. Moreover, there exists a constant $c^x > 0$ such that\[
|Y_s|\leq c^x(1+|X^x_s|).
\]
Then for any control $u\in M^{2}_G(0,\infty)$, we have $J(x, u) \geq \lambda^{\prime} = \lambda$, and the equality holds if and only if
for almost every $t$
\[ f(X^x_t,Z_t)=\kappa(X^x_t,u_t)+R(u_t)Z_t.\]
\end{theorem}

\begin{proof}
It is obvious that $\lambda^{\prime}=\lambda$. Since $(Y,Z,K,\lambda)$ is a solution of the ergodic $G$-BSDE \eqref{AppHMW}, we have
\begin{align*}
Y_{s}  =&Y_T+\int_{s}^{T}[f(X_{r}^{x},Z_{r})-\lambda]dr -\int_{s}^{T}Z_{r}dB_{r}-(K_{T}-K_{s})
\\ =& Y_T+\int_{s}^{T}[f(X_{r}^{x},Z_{r})-\lambda]dr -\int_{s}^{T}Z_{r}dB^u_{r}-\int_{s}^{T}Z_{r}R(u_r)dr-(K_{T}-K_{s}),
\label{AppHMW11}%
\end{align*}
Consequently,
\begin{align*}
\lambda T+ \mathbb{\tilde{E}}^u[ K_{T}]=&\mathbb{\tilde{E}}^u[Y_T-Y_0+\int_{0}^{T}[f(X_{r}^{x},Z_{r})-Z_{r}R(u_r)]dr].
\end{align*}
Note that $\mathbb{\tilde{E}}^u[ K_{T}]=0$, we obtain
\begin{align*}
\lambda &\leq \frac{1}{T}\mathbb{\tilde{E}}^u[Y_T-Y_0+\int_{0}^{T}\kappa(X^x_s,u_s)ds].
\end{align*}
From Remark 5.3 in \cite{HJPS1} and  Lemma \ref{HW2}, we have $\mathbb{\tilde{E}}^u[|Y_T|]\leq C(1+|x|).$ Consequently,
\[
\lim\limits_{T\rightarrow\infty}\frac{1}{T}\mathbb{\tilde{E}}^u[|Y_T-Y_0|]=0.
\]
Thus, we obtain that
\[
J(x,u)=\limsup\limits_{T\rightarrow\infty}\frac{1}{T}\mathbb{\tilde{E}}^{u}[\int^T_0\kappa(X^x_s,u_s)ds]\geq \lambda.
\]
In particular, if $f(X^x_t,Z_t)=\kappa(X^x_t,u_t)+R(u_t)Z_t$, we derive that
\[
\lambda =\limsup\limits_{T\rightarrow\infty}\frac{1}{T}\mathbb{\tilde{E}}^u[Y_T-Y_0+\int_{0}^{T}\kappa(X^x_s,u_s)ds]=J(x,u),
\]
which completes the proof.
\end{proof}
\begin{remark}{\upshape
From the above proof,   if $\limsup$ is changed into $\liminf$ in the equation \eqref{o1},
then the same results hold. Moreover, the optimal value is given by $\lambda$ in both cases.
}
\end{remark}

\appendix

\renewcommand\thesection{Appendix}

\section{ }

\renewcommand\thesection{A}
\subsection{The proof of Lemma \ref{my9}}
\begin{proof} We only prove the first inequality, since the second one can be obtained in a similar way.
Note that $(Y+\bar{K}, Z, K)$ can be seen as the solution to  the following linear $G$-BSDE:
\begin{equation*}
Y_{t}^{\prime}=\xi+\bar{K}_T+\int_{t}^{T}f^{\prime}_{s}ds+\int_{t}^{T}g^{\prime}_{s}d\langle B\rangle_{s}-\int%
_{t}^{T}Z_{s}^{\prime}dB_{s}-(K^{\prime}_{T}-K^{\prime}_{t})
\end{equation*}
with\[
f^{\prime}_s=a_{s}Y_{s}^{\prime}+b_{s}Z_{s}^{\prime}+m_{s}-a_s\bar{K}_s, \ \ g_{s}^{\prime}=c_{s}Y_{s}^{\prime}+d_{s}Z_{s}^{\prime}+n_{s}-c_s\bar{K}_s.
\]
Using Lemma \ref{the5.2}, we conclude that
\begin{align*}
Y_{t}+\bar{K}_t=&(X_{t})^{-1}\mathbb{\hat{E}}_{t}^{\tilde{G}}[X_{T}(\xi+\bar{K}_T)+\int_{t}^{T}%
(m_{s}-a_s\bar{K}_s)X_{s}ds+\int_{t}^{T}(n_{s}-c_s\bar{K}_s)X_{s}d\langle B\rangle_{s}]\\
\leq & (X_{t})^{-1}\mathbb{\hat{E}}_{t}^{\tilde{G}}[X_{T}(\xi+\int_{t}^{T}%
m_{s}X_{s}ds+\int_{t}^{T}n_{s}X_{s}d\langle B\rangle_{s}]\\
& \ \ \ +(X_{t})^{-1}\mathbb{\hat{E}}_{t}^{\tilde{G}}[X_{T}\bar{K}_{T}-\int_{t}^{T}a_{s}\bar{K}_{s}X_{s}%
ds-\int_{t}^{T}c_{s}\bar{K}_{s}X_{s}d\langle
B\rangle_{s}],
\end{align*}
where $X$ is given by \eqref{LSDE2}.
 Then it follows from Lemma \ref{the5.2}
that
\begin{align}\label{my10}
Y_{t}\leq (X_{t})^{-1}\mathbb{\hat{E}}_{t}^{\tilde{G}}[X_{T}\xi+\int_{t}^{T}%
m_{s}X_{s}ds+\int_{t}^{T}n_{s}X_{s}d\langle B\rangle_{s}].
\end{align}

Note that $a_s+\overline{\sigma}^2c_s\leq -\rho_3$ and  $\exp(-\int_{0}^{t}b_{s}d_{s}ds)\mathcal{E}_{t}^{B}\mathcal{E}_{t}^{\tilde{B}}$ is a $\tilde{G}$-martingale, we conclude that
\[
(X_{t})^{-1}\mathbb{\hat{E}}^{\tilde{G}}_{t}[X_{T}]\leq \exp(-\rho_3(T-t)),
\]
which together with inequality \eqref{my10} imply that
\[
Y_t\leq  \rho_1\exp(-\rho_3(T-t))+\frac{1+\overline{\sigma}^2}{\rho_3}(1-\exp(-\rho_3(T-t))\rho_2.
\]
The proof is  complete.
\end{proof}

\subsection{The proof of Lemma \ref{mw12}}
In order to prove Lemma \ref{mw12}, we consider the following type of $G$-FBSDEs with infinite horizon: for each $t\geq 0$ and $\xi\in L^4_G(\Omega_t)$,

\begin{align*} \label{App1m}
\begin{cases}
&X_{s}^{t,\xi}=\xi+\int^s_tb(X_{r}^{t,\xi})dr+\int^s_th_{ij}(X_{r}^{t,\xi})d\langle
B^i,B^j\rangle_{r}+\int^s_t\sigma(X_{r}^{t,\xi})dB_{r},\\&
Y_{s}^{t,\xi}   =Y_{T}^{t,\xi}+\int_{s}^{T}f(X_{r}^{t,\xi}%
,Y_{r}^{t,\xi},Z_{r}^{t,\xi})dr+\int_{s}^{T}g_{ij}(X_{r}^{t,\xi}%
,Y_{r}^{t,\xi},Z_{r}^{t,\xi})d\langle B^i,B^j\rangle_{r}\\
& \ \ \ \ \ \ \ \ \ \ \ \  -\int_{s}^{T}Z_{r}^{t,\xi}dB_{r}-(K_{T}^{t,\xi}-K_{s}^{t,\xi}).
\end{cases}
\end{align*}
 Using the same method as  in Lemma \ref{HM8}, we have the following.
\begin{lemma}\label{HM81}
Under assumptions \emph{(B1)-(B5)}, there exists a constant $M$ depending only on
$L,\alpha_1, \alpha_2,\eta$ and $G$ such that
\[
|Y^{t,\xi}_t-Y^{t,\xi^{\prime}}_t|\leq M|\xi-\xi^{\prime}|.
\]
\end{lemma}
Set
\[
u(t,x):=Y_{t}^{t,x},\ \ (t,x)\in\lbrack0,T]\times\mathbb{R}^{n}.
\]

\begin{lemma}
$u(t,x)$ is a deterministic
function of $(t,x)$. Moreover, $u(t,x)=u(x)$ for each $t\geq 0$.
\end{lemma}
\begin{proof}
Denote by $(Y^{n,x} , Z^{n,x},K^{n,x})$ the unique solution of the
following $G$-BSDE in $[0,n]$:
\begin{align*}
Y^{n,x}_s=\int^n_sf(X^x_r,Y^{n,x}_r,Z^{n,x}_r)dr+\int^n_sg_{ij}(X^x_r,Y^{n,x}_r,Z^{n,x}_r)d\langle B^i,B^j\rangle_r-\int^n_sZ^{n,x}_rdB_r-(K_n^{n,x}-K^{n,x}_s),
\end{align*}
and $(Y^{n,t,x} , Z^{n,t,x},K^{n,t,x})$ the unique solution of the
following $G$-BSDE in $[t,n+t]$:
\begin{align*}
Y^{n,t,x}_s=&\int^{n+t}_sf(X^{t,x}_r,Y^{n,t,x}_r,Z^{n,t,x}_r)dr+\int^{n+t}_sg_{ij}(X^{t,x}_r,Y^{n,t,x}_r,Z^{n,t,x}_r)d\langle B^i,B^j\rangle_r-\int^{n+t}_sZ^{n,t,x}_rdB_r\\&\ -(K_{n+t}^{n,t,x}-K^{n,t,x}_s).
\end{align*}
By the proof of Theorem \ref{HM3}, we get
$\lim\limits_{n\rightarrow\infty}Y^{n,x}_0=u(x)$ and $\lim\limits_{n\rightarrow\infty}Y^{n,t,x}_t=u(t,x)$.
Since $(B_{t+s}-B_t)_{s\geq 0}$ is also  a $G$-Brownian motion, we have
$Y^{n,x}_0=Y^{n,t,x}_t.$
Thus $u(x)=u(t,x)$ and this ends the proof.
\end{proof}

\begin{lemma}
\label{theA.4} For each $\xi\in L_{G}^{4}(\Omega_{t})$, we
have%
\[
u(\xi)=Y_{t}^{t,\xi}.
\]
\end{lemma}
\begin{proof}
By Lemma \ref{HM81}, we only need to prove Lemma \ref{theA.4} for
bounded $\xi\in L_{G}^{4}(\Omega_{t})$. Thus for each
$\varepsilon>0$, we can choose a simple function
$
\eta^{\varepsilon}=\sum_{i=1}^{N}x_{i}\mathbf{1}_{A_{i}},
$
where $(A_{i})_{i=1}^{N}$ is a $\mathcal{B}(\Omega_{t})$-partition and
$x_{i}\in\mathbb{R}^{n}$, such that $|\eta^{\varepsilon}-\xi|\leq\varepsilon$.
It follows from Lemma \ref{HM81} that%
\begin{align*}
|Y_{t}^{t,\xi}-u(\eta^{\varepsilon})|    =|Y_{t}^{t,\xi}-\sum_{i=1}%
^{N}u(x_{i})\mathbf{1}_{A_{i}}|
 =|Y_{t}^{t,\xi}-\sum_{i=1}^{N}Y_{t}^{t,x_{i}}\mathbf{1}_{A_{i}}| =\sum_{i=1}^{N}|Y_{t}^{t,\xi}-Y_{t}^{t,x_{i}}|\mathbf{1}_{A_{i}}
\leq M\varepsilon.
\end{align*}
 Noting that $
|u(\xi)-u(\eta^{\varepsilon})|\leq M\varepsilon,
$
we get $|Y_{t}^{t,\xi}-u(\xi)|\leq2M\varepsilon$. Since
$\varepsilon$ can be arbitrarily small, we obtain $Y_{t}^{t,\xi}=u(\xi)$.
\end{proof}

\begin{proof}[The proof of Lemma \ref{mw12}]
It is easy to check that $X^{t,X^x_t}_s=X^x_s$ for $s\geq t$. Then by the uniqueness of $G$-BSDE \eqref{App1}, we obtain $Y^{t,X^x_t}_t=Y^x_t,$
 which yields the desired result by applying Lemma \ref{theA.4}.
\end{proof}

\subsection{Uniqueness of viscosity solution to fully nonlinear elliptic PDEs}
\begin{theorem}\label{my1}
Under assumptions \emph{(B1)-(B5)}, if $\tilde{u}(x)$ is a bounded continuous viscosity solution to equation \eqref{feynman}, then
\[
u=\tilde{u}.
\]
\end{theorem}
In order to prove Theorem \ref{my1}, we need the following lemmas.
\begin{lemma}\label{my4}
For each bounded and continuous function $\phi\in C_{b}(\mathbb{R}^n)$, $\mathbb{\hat{E}}[\phi(X^x_t)]$ is a continuous function of $(t,x)$.
\end{lemma}
\begin{proof}
Assume $\phi$ is bounded by $M>0$.
For each given $N>0$ and $T>0$, for any $t,t^{\prime}< T$, $x,x^{\prime}\in\mathbb{R}^n$, we have
\begin{align*}
|\mathbb{\hat{E}}[\phi(X^x_t)]-\mathbb{\hat{E}}[\phi(X^{x^{\prime}}_{t^{\prime}})]|&\leq \mathbb{\hat{E}}[|\phi(X^x_t)-\phi(X^{x^{\prime}}_{t^{\prime}})|]\\& \leq
\mathbb{\hat{E}}[|\phi(X^x_t)-\phi(X^{x^{\prime}}_{t^{\prime}})|\mathbf{1}_{\{|X^x_t|\leq N\}\cap \{|X^{x^{\prime}}_{t^{\prime}}|\leq N\}}]\\ & \  +\mathbb{\hat{E}}[|\phi(X^x_t)-\phi(X^{x^{\prime}}_{t^{\prime}})|(\mathbf{1}_{\{|X^x_t|\geq N\}}+\mathbf{1}_{\{|X^{x^{\prime}}_{t^{\prime}}|\geq N\}})]\\&
\leq \mathbb{\hat{E}}[|\phi(X^x_t)-\phi(X^{x^{\prime}}_{t^{\prime}})|\mathbf{1}_{\{|X^x_t|\leq N\}\cap \{|X^{x^{\prime}}_{t^{\prime}}|\leq N\}}]+\frac{2M}{N}(\mathbb{\hat{E}}[|X^{x^{\prime}}_{t^{\prime}}|+|X^x_t|]).
\end{align*}
Note $\phi$ is uniformly continuous on $\{x: |x|\leq N\}$. Then for each given $\epsilon>0$, there is a constant $\rho>0$ such that
\[
|\phi(z)-\phi(z^{\prime})|\leq \frac{\epsilon}{2} \ \text{whenever $|z-z^{\prime}|<\rho$ and $|z|,|z^{\prime}|\leq N$}.
\]
From Proposition 4.1 in \cite{HJPS1}, we obtain
\[
\mathbb{\hat{E}}[|X^x_t-X^{x^{\prime}}_{t^{\prime}}|]\leq C_T(|t-{t^{\prime}}|^{\frac{1}{2}}+|x-x^{\prime}|),
\]
where $C_T$ depends on $L,\alpha_1$, $G$, $n$ and $T$. Then, by Chebyshev's inequality, there is $\delta>0$ such that
\[
c(|X^x_t-X^{x^{\prime}}_{t^{\prime}}|\geq \rho)<\frac{\epsilon}{4M}
\]
whenever $|x-x^{\prime}|\leq \delta$ and $|t-{t^{\prime}}|\leq \delta$.
Consequently,
\begin{align*}
|\mathbb{\hat{E}}[\phi(X^x_t)]-\mathbb{\hat{E}}[\phi(X^{x^{\prime}}_{t^{\prime}})]|&
\leq \mathbb{\hat{E}}[|\phi(X^x_t)-\phi(X^{x^{\prime}}_{t^{\prime}})|\mathbf{1}_{\{|X^x_t-X^{x^{\prime}}_{t^{\prime}}|<\rho\}\cap\{|X^x_t|\leq N\}\cap \{|X^{x^{\prime}}_{t^{\prime}}|\leq N\}}]\\ &\  +\mathbb{\hat{E}}[|\phi(X^x_t)-\phi(X^{x^{\prime}}_t)|\mathbf{1}_{\{|X^x_t-X^{x^{\prime}}_{t^{\prime}}|\geq \rho\}}]
+\frac{2M}{N}(\mathbb{\hat{E}}[|X^{x^{\prime}}_{t^{\prime}}|+|X^x_t|])\\&
\leq \epsilon+\frac{2M}{N}(\mathbb{\hat{E}}[|X^{x^{\prime}}_{t^{\prime}}|+|X^x_t|])
\end{align*}
whenever $|x-x^{\prime}|\leq \delta$ and $|t-{t^{\prime}}|\leq \delta$. Thus we get
\[
\limsup\limits_{(t^{\prime},x^{\prime})\rightarrow(t,x)}|\mathbb{\hat{E}}[\phi(X^x_t)]-\mathbb{\hat{E}}[\phi(X^{x^{\prime}}_{t^{\prime}})]| \leq \epsilon+\frac{2M}{N}(\mathbb{\hat{E}}[|X^{x^{\prime}}_{t^{\prime}}|+|X^x_t|]).
\]
The proof is complete by letting $\epsilon \downarrow 0$ and then $N\rightarrow\infty$.
\end{proof}

Now we consider the following type of $G$-BSDEs on $[0,T]$ with $T>0$: for each $t\in[0,T]$ and $x\in \mathbb{R}^n$,
\begin{align}
Y_{s}^{t,T,x}   =&\phi(X^{t,x}_T)+\int_{s}^{T}f(X_{r}^{t,x}%
,Y_{r}^{t,T,x},Z_{r}^{t,T,x})dr+\int_{s}^{T}g_{ij}(X_{r}^{t,x}%
,Y_{r}^{t,T,x},Z_{r}^{t,T,x})d\langle B^i,B^j\rangle_{r}\nonumber\\
& \ \ \ \ \ \ \ \ \ \ \ \  -\int_{s}^{T}Z_{r}^{t,T,x}dB_{r}-(K_{T}^{t,T,x}-K_{s}^{t,T,x}),\label{my2}
\end{align}
where $\phi$ is a continuous function bounded by $M>0$. In particular, denote $(Y^{T,x},Z^{T,x},K^{T,x})=(Y^{0,T,x},Z^{0,T,x},K^{0,T,x})$.
Then we denote $\bar{u}(t,x)=Y^{t,T,x}_t$.
Note that there exists a sequence  Lipschitz functions $\{\phi^m\}_{m=1}^{\infty}$ bounded by $M$ such that
\[
|\phi(x)-\phi^m(x)|\leq \frac{1}{m}\mathbf{1}_{\{|x|\leq m\}}+2M\mathbf{1}_{\{|x|>m\}}.
\]
Then let $(Y^{t,T,m,x},Z^{t,T,m,x},K^{t,T,m,x})$ be the unique $\mathfrak{S}_{G}^{2}(0,T)$-solution of $G$-FBSDEs \eqref{my2} with terminal condition
$Y^{t,T,m,x}_T=\phi^m(X^{t,x}_T)$ and denote $\bar{u}^m(t,x)=Y^{t,T,m,x}_t$.

\begin{lemma} [\cite{HJPS1}]\label{my3}
Under assumptions \emph{(B1)} and \emph{(B2)}, $\bar{u}^m(t,x)$ is the unique viscosity solution of the following fully nonlinear PDE with terminal condition $\bar{u}^m(T,x)=\phi^m(x)$:
\begin{equation}
\left\{
\begin{array}
[c]{l}%
\partial_{t}u+ G(H(D_{x}^{2}u,D_{x}u,u,x))+\langle
b(x),D_{x}u\rangle
 +f(x,u,D_{x}u\sigma(x))=0,\\
u(T,x)=\phi(x).
\end{array}
\right.  \label{myfeynman}%
\end{equation}
Moreover, $\bar{u}^m(t,X^x_t)=Y^{T,m,x}_t$ for each $t\in[0,T]$.
\end{lemma}
\begin{lemma}\label{my7}
Assume \emph{(B1)} and \emph{(B2)} hold. Then we have
\begin{description}
\item[(1)] There exists a constant $C$ depending on $M$, $T$, $G$, $L$, $\alpha$ and $\alpha_2$ such that \[\|Y^{t,T,m,x}\|_{S_G^2(t,T)}+\|Z^{t,T,m,x}\|_{M_G^2(t,T)}+\|Y^{t,T,x}\|_{S_G^2(t,T)}+\|Z^{t,T,x}\|_{M_G^2(t,T)}\leq C, \ \forall x\in\mathbb{R}^n,m\geq 1;\]
\item[(2)] $\lim\limits_{m\rightarrow\infty}\mathbb{\hat{E}}[\sup\limits_{s\in[t,T]}|Y^{t,T,m,x}_s-Y^{t,T,x}_s|^2]=0$;
\item[(3)] $\bar{u}(t,x)$ is a bounded and continuous function;
\item[(4)] $\lim\limits_{m\rightarrow \infty}\bar{u}^m(t_m,x_m)=\bar{u}(t,x)$ for each given $(t_m,x_m)\in[0,T]\times\mathbb{R}^n$ with $(t_m,x_m)\rightarrow (t,x)$.
\end{description}
\end{lemma}
\begin{proof}
Note that $\phi^m$ and $f(x,0,0),g_{ij}(x,0,0)$ are uniformly bounded. Applying Proposition 3.5 and Corollary 5.2 in \cite{HJPS}, we obtain (1).
By Theorem \ref{pro3.5} and Theorem 3.3 in \cite{Song11}, we can find a constant $\tilde{C}$  depending on $M$, $T$, $G$, $L$, $\alpha$ and $\alpha_2$ (may vary from line to line), such that,
\begin{align}
\lim\limits_{m\rightarrow\infty}\mathbb{\hat{E}}[\sup\limits_{s\in[t,T]}|Y^{t,T,m,x}_s-Y^{t,T,x}_s|^2]\leq &\lim\limits_{m\rightarrow\infty}\tilde{C} ((\mathbb{\hat{E}}[|\phi(X^{t,x}_T)-\phi^m(X^{t,x}_T)|^{3}])^{\frac{2}{3}}+\mathbb{\hat{E}}[|\phi(X^{t,x}_T)-\phi^m(X^{t,x}_T)|^{3}])\nonumber\\
\leq &\lim\limits_{m\rightarrow\infty}\tilde{C}( \frac{1}{m^2}+\frac{\mathbb{\hat{E}}[|X^{t,x}_T|^{3}]+(\mathbb{\hat{E}}[|X^{t,x}_T|^{3}])^{\frac{2}{3}}}{m^2})=0.\label{my5}
\end{align}
In particular, $\lim\limits_{m\rightarrow\infty}\bar{u}^m(t,x)=\bar{u}(t,x)$.

Now we prove $\lim\limits_{m\rightarrow \infty}\bar{u}(t_m,x_m)=\bar{u}(t,x)$ for each given $(t_m,x_m)\in[0,T]\times\mathbb{R}^n$ with $(t_m,x_m)\rightarrow (t,x)$.
Without loss of generality, we assume $t_m\leq t$ and $g_{ij}=0$. Using the method as in (2) and Lemma \ref{my4}, we can obtain
\begin{align}\label{my6}
\lim\limits_{m\rightarrow\infty}\mathbb{\hat{E}}[\sup\limits_{s\in[t,T]}|Y^{t_m,T,x_m}_s-Y^{t,T,x}_s|^2+\int^T_t|Z^{t_m,T,x_m}_s-Z^{t,T,x}_s|^2ds]=0. \end{align}
By equation \eqref{my2}, we have
\[\bar{u}(t,x)+(K_{T}^{t,T,x}-K_{t}^{t,T,x})=\phi(X^{t,x}_T)+\int_{t}^{T}f(X_{r}^{t,x}%
,Y_{r}^{t,T,x},Z_{r}^{t,T,x})dr-\int_{t}^{T}Z_{r}^{t,T,x}dB_{r}.\]
Taking expectation on both sides yields that
\[
\bar{u}(t,x)=\mathbb{\hat{E}}[\phi(X^{t,x}_T)+\int_{t}^{T}f(X_{r}^{t,x}%
,Y_{r}^{t,T,x},Z_{r}^{t,T,x})dr].
\]
Consequently,
\begin{align*}
|\bar{u}(t,x)-\bar{u}(t_m,x_m)|&
\leq \mathbb{\hat{E}}[|\phi(X^{t,x}_T)-\phi(X^{t_m,x_m}_T)|+\int_{{t_m}}^{{t}}|f(X_{r}^{t_m,x_m}%
,Y_{r}^{t_m,T,x_m},Z_{r}^{t_m,T,x_m})|\,dr\\
&\ \ \  +\int_{{t}}^{T}|f(X_{r}^{t,x}%
,Y_{r}^{t,T,x},Z_{r}^{t,T,x})-f(X_{r}^{t_m,x_m}%
,Y_{r}^{t_m,T,x_m},Z_{r}^{t_m,T,x_m})|\,dr]\\
& \leq  \mathbb{\hat{E}}[(t-t_m)^{\frac{1}{2}}(\int_{{t_m}}^{{t}}3(|f(X_r^{t_m,x_m},0,0)|^2+|LY_{r}^{t_m,T,x_m}|^2+|\alpha_2Z_{r}^{t_m,T,x_m}|^2)\,dr)^{\frac{1}{2}}\\
&\ \ \ +\int_{{t}}^{T}(L|X_{r}^{t,x}-X_{r}^{t_m,x_m}|+L
|Y_{r}^{t,T,x}-Y_{r}^{t_m,T,x_m}| +\alpha_2|Z_{r}^{t,T,x}-Z_{r}^{t_m,T,x_m}|)\,dr\\& \ \ \ +
|\phi(X^{t,x}_T)-\phi(X^{t_m,x_m}_T)|].
\end{align*}
By Lemma \ref{my4}, (1) and equation \eqref{my6}, we derive that
\[
\lim\limits_{m\rightarrow\infty}|\bar{u}(t,x)-\bar{u}(t_m,x_m)|=0,
\]
 and $u$ is a bounded continuous function.

 From (3), we get that
\begin{align*}\lim\limits_{m\rightarrow\infty}|\bar{u}^m(t_m,x_m)-\bar{u}(t,x)|\leq &\lim\limits_{m\rightarrow\infty}|\bar{u}^m(t_m,x_m)-\bar{u}(t_m,x_m)|+ \lim\limits_{m\rightarrow\infty}|\bar{u}(t_m,x_m)-\bar{u}(t,x)|\\=&\lim\limits_{m\rightarrow\infty}|\bar{u}^m(t_m,x_m)-\bar{u}(t_m,x_m)|.
\end{align*}
By equation \eqref{my5}, we obtain
\[
\lim\limits_{m\rightarrow\infty}|\bar{u}^m(t_m,x_m)-\bar{u}(t,x)|\leq\lim\limits_{m\rightarrow\infty}\tilde{C}( \frac{1}{m}+\frac{\mathbb{\hat{E}}[|X^{t_m,x_m}_T|^{3}]^{\frac{1}{2}}+\mathbb{\hat{E}}[|X^{t_m,x_m}_T|^{3}]^{\frac{1}{3}}}{m})=0.
\]
 The proof is complete.
\end{proof}

By Lemmas \ref{my3}, \ref{my7},   Theorem 6.1 in \cite{BJ} and Proposition 4.3 in \cite{CMI},
we have the following result, which is the nonlinear Feynman-Kac formula for parabolic PDE.

\begin{lemma} \label{my8}
Under assumptions \emph{(B1)} and \emph{(B2)}, $\bar{u}(t,x)$ is the unique viscosity solution of the  fully nonlinear PDE
\eqref{myfeynman} with terminal condition $\bar{u}(T,x)=\phi(x)$. In particular, $\bar{u}(t,X^x_t)=Y^{T,x}_t$
\end{lemma}

Now we give the proof of Theorem \ref{my1}.

\begin{proof}[The  proof of Theorem \ref{my1}]
For each $T>0$, by  the definition of viscosity solution, we obtain $\tilde{u}$ is the unique viscosity solution of the  fully nonlinear PDE
\eqref{myfeynman} with terminal condition $\phi(x)=\tilde{u}(x)$.
Then it follows Lemma \ref{my8}, $\tilde{u}(X^x_t)=Y^{T,x}_t$ for each $t\in[0,T]$, where
\begin{align}
Y_{s}^{T,x}   =&\tilde{u}(X^{x}_T)+\int_{s}^{T}f(X_{r}^{x}%
,Y_{r}^{T,x},Z_{r}^{t,T,x})dr+\int_{s}^{T}g_{ij}(X_{r}^{x}%
,Y_{r}^{T,x},Z_{r}^{T,x})d\langle B^i,B^j\rangle_{r}\nonumber\\
& \ \ \ \ \ \ \ \ \ \ \ \  -\int_{s}^{T}Z_{r}^{T,x}dB_{r}-(K_{T}^{T,x}-K_{s}^{T,x}).
\end{align}
By the uniqueness of solution to $G$-BSDE in finite horizon,
it is obvious $(Z_{r}^{x,T},K_{r}^{x,T})=(Z_{r}^{x,S},K_{r}^{x,S})$ for $S>T$.
Set $(Z^x_t,K^x_t)=(Z^{x,T}_t,K^{x,T}_t)$ for some $T\geq t$. Then $(\tilde{u}(X^x_{t}),Z^x_t,K^x_t)_{t\geq 0}$
satisfies equation \eqref{App1}. Applying Theorem \ref{HM3}, we obtain $\tilde{u}(X^x_{t})={u}(X^x_{t})$. In particular,
$\tilde{u}(x)=u(x)$, which is the desired result.
\end{proof}
\begin{remark}{
\upshape
In this section, we introduce a new  method to  prove the uniqueness of the
viscosity solutions to   elliptic PDEs in $\mathbb{R}^n$, which non-trivially generalize the
ones of \cite{PE} for fully nonlinear case. In particular, this method can be applied to deal with more general elliptic PDEs, for example,
the usual HJB equations.
}
\end{remark}

\end{document}